\documentclass[12pt]{amsart}
\usepackage{amsmath}
\usepackage{amssymb}
\usepackage{latexsym}

\newcommand{\halmos}{\rule{5pt}{5pt}}

\numberwithin{equation}{section}

\newtheorem{theorem}{Theorem}[section]
\newtheorem{prop}[theorem]{Proposition}

\theoremstyle{definition}

\begin{document}
\title[$q$-Heun equation  and initial-value space of $q$-Painlev\'e equation]
{$q$-Heun equation  and initial-value space of $q$-Painlev\'e equation}
\author{Shoko Sasaki}
\address{Address of S.S.,~S.T., Department of Mathematics, Faculty of Science and Engineering, Chuo University, 1-13-27 Kasuga, Bunkyo-ku, Tokyo 112-8551, Japan}
\author{Shun Takagi}
\author{Kouichi Takemura}
\address{Address of K.T., Department of Mathematics, Ochanomizu University, 2-1-1 Otsuka, Bunkyo-ku, Tokyo 112-8610, Japan}
\email{takemura.kouichi@ocha.ac.jp}
\begin{abstract}
We show that the $q$-Heun equation and its variants appear in the linear $q$-difference equations associated to some $q$-Painlev\'e equations by considering the blow-up associated to their initial-value spaces.
We obtain the firstly degenerated Ruijsenaars-van Diejen operator from the linear $q$-difference equation associated to the $q$-Painlev\'e equation of type $E_8$.
\end{abstract}
\maketitle

\section{Introduction}
The Painlev\'e transcendent is a transcendental solution of one of the six Painlev\'e equations and it is potentially a member of prospective special functions.
The sixth Painlev\'e equation is a non-linear generalization of the hypergeometric equation, and the hypergeometric equation is a standard form of the second order Fuchsian differential equation with three singularities $\{ 0,1,\infty \}$.
Heun's differential equation is a standard form of the second order Fuchsian differential equation with four singularities $\{ 0,1,t,\infty \}$, and it is written as
\begin{equation*}
\frac{d^2y}{dx^2}  +  \left( \frac{\gamma}{x}+\frac{\delta }{x-1}+\frac{\epsilon}{x-t}\right) \frac{dy}{dx} +\frac{\alpha \beta x -B}{x(x-1)(x-t)} y =  0 ,
\end{equation*}
under the condition $\gamma +\delta +\epsilon =\alpha +\beta +1$.
The parameter $B$, which is independent from the local exponents, is called the accessory parameter.
It is known that Heun's differential equation is related with the sixth Painlev\'e equation (see \cite{DK}, \cite{TakMH} and references therein).

A $q$-difference analogue of Heun's differential equation was given by Hahn \cite{Hahn}, and we may write it as 
\begin{align}
    &(x-h_{1}q^{1/2})(x-h_{2}q^{1/2})g(x/q)+l_{3}l_{4}(x-l_{1}q^{-1/2})(x-l_{2}q^{-1/2})g(qx) \label{eq:qHeun} \\
    &-\{(l_{3}+l_{4})x^{2}+Ex+(l_{1}l_{2}l_{3}l_{4}h_{1}h_{2})^{1/2}(h_{3}^{1/2}+h_{3}^{-1/2})\}g(x)=0 \, . \nonumber
\end{align}
Note that the $q$-Heun equation was used to introduced the variants of the $q$-hypergeometric equation \cite{HMST}.
Several discrete Painlev\'e equations had been studied in the 1990's.
Sakai \cite{Sak} proposed a list of the second order discrete Painlev\'e equations.
In Sakai's list, each member of the $q$-difference Painlev\'e equation was labelled by the affine root system from the symmetry.
A relationship between the $q$-Heun equation and the $q$-Painlev\'e equation of type $D^{(1)}_5$ ($q$-$P(D^{(1)}_{5})$) was pointed out in \cite{TakR}, and it was called the $q$-Painlev\'e-Heun correspondence.

In this paper, we establish a further relationship between the $q$-Heun equation and the $q$-Painlev\'e equation by using the initial-value space of the $q$-Painlev\'e equation.
Namely we derive the $q$-Heun equation by considering the Lax pair and the initial-value space.

Jimbo and Sakai \cite{JS} introduced a $q$-analogue of the sixth Painlev\'e equation by considering the connection preserving deformation of certain $q$-linear difference equations.
It is essentially equivalent to a compatibility condition of a Lax pair, and the Lax pair, which was reformulated by Kajiwara, Noumi and Yamada \cite{KNY}, is written as 
\begin{align}
L_{1} &= \Bigl\{\frac{z(g\nu_{1}-1)(g\nu_{2}-1)}{qg}-\frac{\nu_{1}\nu_{2}\nu_{3}\nu_{4}(g- \nu_{5}/\kappa_{2})(g- \nu_{6}/\kappa_{2})}{fg}\Bigr\} \nonumber \\
&+\frac{\nu_{1}\nu_{2}(z- q \nu_{3})(z - q \nu_{4})}{q(qf-z)}(g-T^{-1}_{z})+\frac{(z- \kappa_{1} /\nu_{7})(z - \kappa_{1}/\nu_{8})}{q(f-z)}\Bigl(\frac{1}{g} - T_{z}\Bigr) , \nonumber \\
L_{2} &= \Bigl(1-\frac{f}{z} \Bigr)T+T_{z}-\frac{1}{g} . \nonumber 
\end{align}
Here $T_{z}$ represents the transformation $z\rightarrow qz$ and $T$ represents the time evolution.
The parameters are constrained by the relation $\kappa_{1} ^2 \kappa_{2} ^2 =q \nu _1 \nu _2 \dots \nu _8$. 
We use the notation such as $T(f)=\overline{f}$ and $T^{-1}(g)=\underline{g}$.
The time evolution of the parameters is given by $\overline{\kappa }_1= \kappa_{1} /q$, $\overline{\kappa }_2= q \kappa_{2}$ and $ \overline{\nu }_i= \nu _{i}$ $(i=1,2,\dots ,8)$.
It follows from the compatibility condition for the Lax operators $L_1$ and $L_2$ (see \cite{KNY}) that the parameters $f$ and $g$ satisfies
\begin{equation}
f\overline{f} = \nu_3\nu_4\, \frac{ ( g - \nu_5 /\kappa_2 ) ( g - \nu_6/\kappa_2 )}{ ( g - 1/\nu_1 ) ( g - 1/\nu_2) }, \quad
g\underline{g} = \frac{1}{\nu_1\nu_2}\frac{( f - \kappa_1 /\nu_7 ) ( f - \kappa_1/\nu_8 )}{(f - \nu_3)(f - \nu_4)}. \nonumber 
\end{equation}
This is the $q$-Painlev\'e equation essentially introduced by Jimbo and Sakai \cite{JS}, which we denote by $q$-$P(D^{(1)}_{5})$.
The time evolution of the variable $f$ is given by
\begin{equation}
\overline{f} = \nu_3\nu_4\, \frac{ ( g - \nu_5 /\kappa_2 ) ( g - \nu_6/\kappa_2 )}{ f( g - 1/\nu_1 ) ( g - 1/\nu_2) } \label{eq:ftimeevol}
\end{equation}
and the value $\overline{f} $ is indefinite in the case $(f,g) = (0, \nu_5 /\kappa_2)$ because of the form $0/0$ in the right hand side of Eq.~(\ref{eq:ftimeevol}).
To resolve the indefiniteness, we consider the blow-up at $(f,g) = (0, \nu_5 /\kappa_2 )$.
Namely we set $(f,g)=(f_{1},f_{1}g_{1}+ \nu_5 /\kappa_2 )$. Then the point $(f,g)=(0, \nu_5 /\kappa_2)$ corresponds to the line $f_{1}=0$, and we have
\begin{equation}
\overline{f} = \, \frac{\nu_3\nu_4  g_{1} ( f_{1}g_{1}+ ( \nu_5 - \nu_6 )/\kappa_2 )}{ ( f_{1}g_{1}+ \nu_5 /\kappa_2 - 1/\nu_1) ( f_{1}g_{1}+ \nu_5/\kappa_2 - 1/\nu_2 ) } . \nonumber 
\end{equation}
Hence the indefiniteness is resolved.
The indefiniteness may occur at the following eight points;
\begin{equation}
P_{i}:(\infty,{1}/\nu_{i})_{i=1,2},\quad (\nu_{i},\infty)_{i=3,4},\quad (0,\nu_{i}/\kappa_{2})_{i=5,6},\quad (\kappa_{1}/\nu_{i},0)_{i=7,8}. \nonumber 
\end{equation}
The space of initial values of $q$-$P(D^{(1)}_5)$ is defined by blowing up eight points $P_1, \dots ,P_8$ of ${\mathbb P}^1 \times {\mathbb P}^1$.

We now consider the blow-up at $P_{5} $ $(0,\nu_{5}/\kappa_{2})$ jointly with the $q$-difference equation $L_1 y(z)=0$, where $L_1$ is one of the Lax pair.
Note that the equation $L_1 y(z)=0$ itself is different from the $q$-Heun equation in Eq.~(\ref{eq:qHeun}).
We substitute $(f,g)=(f_{1},f_{1}g_{1}+\nu_{5}/\kappa_{2})$ into the equation $L_{1} y(z) = 0 $ and set $f_1 =0$.
Then we have
\begin{align}
    &(z-q\nu_{3})(z-q\nu_{4})y(z/q)+\frac{1}{\nu_{1}\nu_{2}}\Bigl(z-\frac{\kappa_{1}}{\nu_{7}}\Bigr) \Bigl(z-\frac{\kappa_{1}}{\nu_{8}}\Bigr)y(qz) \nonumber \\
    &-\Bigl[\Bigl( \frac{1}{\nu_{1}}+\frac{1}{\nu_{2}}\Bigr)z^{2} +\Bigl\{ q g_{1}\nu_{3}\nu_{4} \Bigl( 1 -\frac{\nu_{6}}{\nu_{5}} \Bigr) -\frac{q\nu_{5}(\nu_{3} + \nu_{4} )}{\kappa_{2}}-\frac{\kappa_{1}\kappa_{2}(\nu_{7} +\nu_{8} )}{\nu_{1}\nu_{2}\nu_{5}\nu_{7}\nu_{8}}\Bigr \} z \nonumber \\
    &\qquad +\frac{q\kappa_{1}{\nu_{3}^{1/2}{\nu_{4}}^{1/2}}}{{\nu_{1}}^{1/2}{\nu_{2}}^{1/2}{\nu_{7}}^{1/2}{\nu_{8}}^{1/2}}\Bigl\{\Bigl(\frac{q\nu_{6}}{\nu_{5}}\Bigr)^{1/2}+\Bigl(\frac{q\nu_{6}}{\nu_{5}}\Bigr)^{-1/2}\Bigr\}\Bigr]y(z)=0  \nonumber 
\end{align}
in \S \ref{sec:D15P5P6}.
Thus we obtain the $q$-Heun equation with the accessory parameter by considering the equation $L_1 y(z)=0$ on the line $f_1=0$, which is an exceptional curve with respect to the blow-up at the point $P_5$ on the ininial value space.
We can also obtain the $q$-Heun equation by the blow-up at $P_i$ $(i=1,\dots , 8)$. 

We establish similar results on the $q$-Painlev\'e equation of type $E^{(1)}_6$, $E^{(1)}_7$ and $E^{(1)}_8$ in this paper.
The corresponding linear $q$-difference equations for the cases $E^{(1)}_6$ and $E^{(1)}_7$ are the variants of $q$-Heun equation \cite{TakqH} (see Eqs.~(\ref{eq:qHeunE16}), (\ref{eq:qHeunE17})).
Note that $q$-Heun equation and the variants are also obtained by degenerations of Ruijsenaars-van Diejen system \cite{RuiN,vD,TakR}.
The corresponding linear $q$-difference equations for the case $E^{(1)}_8$ is the firstly degenerated Ruijsenaars-van Diejen operator.
Note that Noumi, Ruijsenaars and Yamada \cite{NRY} obtained the non-degenerate Ruijsenaars-van Diejen system from the Lax formalism of the elliptic Painlev\'e equation $e$-$P(E^{(1)}_8)$. 

This paper is organized as follows.
In section \ref{sec:ini}, we recall the $q$-Painlev\'e equations $q$-$P(D^{(1)}_5)$, $q$-$P(E^{(1)}_6)$ and $q$-$P(E^{(1)}_7)$, their Lax pairs, and the corresponding initial-value spaces by following the review by Kajiwara, Noumi and Yamada \cite{KNY}.
We consider restriction of the variables of initial-value spaces and obtain the $q$-Heun equation and the variants.
In section \ref{sec:qPE8}, we recall the $q$-Painlev\'e equation of type $E^{(1)}_8$ by following Yamada \cite{Y}, and obtain the firstly degenerated Ruijsenaars-van Diejen operator by considering a suitable limit.

\section{$q$-Painlev\'e equation, Lax pairs, initial-value space of $q$-Painlev\'e equation and $q$-Heun equation} \label{sec:ini}

Lax pairs and the initial-value spaces for discrete Painlev\'e equations were reviewed by Kajiwara, Noumi and Yamada in \cite{KNY}.
We recall them for the $q$-Painlev\'e equations of type $D^{(1)}_5$, $E^{(1)}_6$ and $E^{(1)}_7 $, and we apply the procedure of the blow-up associated with the the initial-value spaces to the linear $q$-difference equation related to the Lax pairs.

Throughout this section, we assume that the parameters $ \kappa_{1} , \kappa_{2} , \nu _1 , \nu _2 , \dots ,\nu _8$ are non-zero and satisfy the relation $ \kappa_{1} ^2 \kappa_{2} ^2 =q \nu _1 \nu _2 \dots \nu _8$. 

\subsection{The case $D^{(1)}_5$} $ $ 

The $q$-Painlev\'e equation of type $D^{(1)}_5 $ ($q$-$P(D^{(1)}_5)$) was given as
\begin{equation}
f\overline{f} = \nu_3\nu_4\, \frac{ ( g - \nu_5/\kappa_2 ) ( g - \nu_6/ \kappa_2 )}{( g - 1/\nu_1 ) ( g - 1/\nu_2 ) }, \quad
g\underline{g} = \frac{1}{\nu_1\nu_2}\frac{( f - \kappa_1 /\nu_7 ) ( f - \kappa_1/\nu_8 )}{(f - \nu_3)(f - \nu_4)}. \nonumber 
\end{equation}
The space of initial condition for $q$-$P(D^{(1)}_5)$ is realized by blowing up eight points $P_{1}, \dots , P_8 $ of ${\mathbb P^1}\times {\mathbb P^1}$, where 
\begin{equation}
P_{i}:(\infty,{1}/\nu_{i})_{i=1,2},\quad (\nu_{i},\infty)_{i=3,4},\quad (0,\nu_{i}/\kappa_{2})_{i=5,6},\quad (\kappa_{1}/\nu_{i},0)_{i=7,8} \nonumber
\end{equation}
On the other hand, $q$-$P(D^{(1)}_5)$ is realized by the compatibility condition for the Lax pair $L_1$ and $L_2$, where
\begin{align}
L_{1} &= \Bigl\{\frac{z(g\nu_{1}-1)(g\nu_{2}-1)}{qg}-\frac{\nu_{1}\nu_{2}\nu_{3}\nu_{4}(g- \nu_{5}/\kappa_{2})(g- \nu_{6}/\kappa_{2})}{fg}\Bigr\} \nonumber \\
&+\frac{\nu_{1}\nu_{2}(z- q \nu_{3})(z - q \nu_{4})}{q(qf-z)}(g-T^{-1}_{z})+\frac{(z- \kappa_{1} /\nu_{7})(z - \kappa_{1}/\nu_{8})}{q(f-z)}\Bigl(\frac{1}{g} - T_{z}\Bigr) \nonumber \\
L_{2} &= \Bigl(1-\frac{f}{z} \Bigr)T+T_{z}-\frac{1}{g} \nonumber 
\end{align}
We track the linear differential equation $L_1 y(z)= 0$ on the process of the blow-up of the point $P_i$ $(i=1,2,\dots ,8)$ respectively, and we investigate a relationship with the $q$-Heun equation
\begin{align}
    &(x-h_{1}q^{1/2})(x-h_{2}q^{1/2})g(x/q)+l_{3}l_{4}(x-l_{1}q^{-1/2})(x-l_{2}q^{-1/2})g(qx) \label{eq:qHeunD15}\\
    &-\{(l_{3}+l_{4})x^{2}+Ex+(l_{1}l_{2}l_{3}l_{4}h_{1}h_{2})^{1/2}(h_{3}^{1/2}+h_{3}^{-1/2})\}g(x) =0 .\nonumber 
\end{align}
The parameter $E$ is called the accessory parameter.
Note that detailed calculation in the following subsections was performed in \cite{Ssk}.

\subsubsection{The points $P_{5}:(0,\nu_{5}/\kappa_{2})$ and $P_{6}:(0,\nu_{6}/\kappa_{2})$} \label{sec:D15P5P6} $ $

We realize the blow-up at $(f,g)=(0,\nu_{5}/\kappa_{2})$ by setting $(f,g)=(f_{1},f_{1}g_{1}+\nu_{5}/\kappa_{2})$.
Then the point $P_{5}:(0,\nu_{5}/\kappa_{2})$ corresponds to the line $f_{1}=0$.
By the transformation $(f,g)=(f_{1},f_{1}g_{1}+\nu_{5}/\kappa_{2})$, the operator $L_{1}$ is written as 
\begin{align}
    L_{1}=&-\frac{g_{1}\nu_{1}\nu_{2}\nu_{3}\nu_{4}(f_{1}g_{1}\kappa_{2}+\nu_{5}-\nu_{6})}{f_{1}g_{1}\kappa_{2}+\nu_{5}}+\frac{\nu_{1}\nu_{2}(q\nu_{3}-z)(q\nu_{4}-z)(f_{1}g_{1}\kappa_{2}+\nu_{5}-\kappa_{2}T_{z}^{-1})}{\kappa_{2}q(f_{1}q-z)} \nonumber \\
    &+\frac{(\nu_{1}\nu_{5}-\kappa_{2}+f_{1}g_{1}\kappa_{2}\nu_{1})(\nu_{2}\nu_{5}-\kappa_{2}+f_{1}g_{1}\kappa_{2}\nu_{2})z}{\kappa_{2}q(f_{1}g_{1}\kappa_{2}+\nu_{5})} \nonumber \\
    &-\frac{(-\kappa_{1}+\nu_{7}z)(-\kappa_{1}+\nu_{8}z)(-\kappa_{2}+( f_{1}g_{1}\kappa_{2}+\nu_{5})T_{z})}{q(f_{1}g_{1}\kappa_{2}+\nu_{5})\nu_{7}\nu_{8}(f_{1}-z)}. \nonumber 
\end{align}
We set $f_1=0$ on the equation $L_1 y(z)=0 $.
Then we have
\begin{align}
    &\nu_{1}\nu_{2}(z-q\nu_{3})(z-q\nu_{4})y(z/q)+\Bigl(z-\frac{\kappa_{1}}{\nu_{7}}\Bigr) \Bigl(z-\frac{\kappa_{1}}{\nu_{8}}\Bigr)y(qz) \label{eq:D5P5} \\
    &-\Bigl[ ( \nu_{1}+ \nu_{2} ) z^{2} +\Bigl\{ q g_{1}\nu_{1}\nu_{2}\nu_{3}\nu_{4} \Bigl( 1 -\frac{\nu_{6}}{\nu_{5}} \Bigr) -\frac{q\nu_{1}\nu_{2}\nu_{5}(\nu_{3} + \nu_{4} )}{\kappa_{2}}-\frac{\kappa_{1}\kappa_{2}(\nu_{7} +\nu_{8} )}{\nu_{5}\nu_{7}\nu_{8}}\Bigr \} z \nonumber \\
    &\qquad +\frac{q^{2}\nu_{1}\nu_{2}\nu_{3}\nu_{4}\nu_{5}}{\kappa_{2}}+\frac{{\kappa_{1}}^{2}\kappa_{2}}{\nu_{5}\nu_{7}\nu_{8}}\Bigr\}\Bigr]y(z)=0 . \nonumber 
\end{align}
It follows from the relation $\kappa_{1} ^2 \kappa_{2} ^2 = q \nu_{1} \nu_{2} \dots \nu_{8} $ that Eq.~(\ref{eq:D5P5}) is written as
\begin{align}
    &(z-q\nu_{3})(z-q\nu_{4})y(z/q)+\frac{1}{\nu_{1}\nu_{2}}\Bigl(z-\frac{\kappa_{1}}{\nu_{7}}\Bigr) \Bigl(z-\frac{\kappa_{1}}{\nu_{8}}\Bigr)y(qz) \nonumber \\
    &-\Bigl[\Bigl( \frac{1}{\nu_{1}}+\frac{1}{\nu_{2}}\Bigr)z^{2} +\Bigl\{ q g_{1}\nu_{3}\nu_{4} \Bigl( 1 -\frac{\nu_{6}}{\nu_{5}} \Bigr) -\frac{q\nu_{5}(\nu_{3} + \nu_{4} )}{\kappa_{2}}-\frac{\kappa_{1}\kappa_{2}(\nu_{7} +\nu_{8} )}{\nu_{1}\nu_{2}\nu_{5}\nu_{7}\nu_{8}}\Bigr \} z \nonumber \\
    &\qquad +\frac{q\kappa_{1}{\nu_{3}^{1/2}{\nu_{4}}^{1/2}}}{{\nu_{1}}^{1/2}{\nu_{2}}^{1/2}{\nu_{7}}^{1/2}{\nu_{8}}^{1/2}}\Bigl\{\Bigl(\frac{q\nu_{6}}{\nu_{5}}\Bigr)^{1/2}+\Bigl(\frac{q\nu_{6}}{\nu_{5}}\Bigr)^{-1/2}\Bigr\}\Bigr]y(z)=0 . \nonumber 
\end{align}
We compare it with the $q$-Heun equation given in Eq.~(\ref{eq:qHeunD15}).
Then we obtain the following correspondence
\begin{align}
    &h_{1}=q^{1/2}\nu_{3},\quad h_{2}=q^{1/2}\nu_{4},\quad h_{3}=\frac{q\nu_{6}}{\nu_{5}} \nonumber \\
    &l_{1}=\frac{q^{1/2}\kappa_{1}}{\nu_{7}},\quad l_{2}=\frac{q^{1/2}\kappa_{1}}{\nu_{8}},\quad l_{3}=\frac{1}{\nu_{1}},\quad l_{4}=\frac{1}{\nu_{2}}. \nonumber 
\end{align}
We may regard the parameter $g_1$ as an accessory parameter.

On the blow-up at the point $P_{6}:(0, \nu_{6}/\kappa_{2})$, we obtain the $q$-Heun equation where the parameters $\nu _5$ and $\nu _6$ are exchanged.

\subsubsection{The points $P_{7}:(\kappa_{1}/\nu_{7},0)$ and $P_{8}:(\kappa_{1}/\nu_{8},0)$} $ $

We realize the blow-up at $(f,g)=(\kappa_{1}/\nu_{7},0)$ by setting $(f,g)=(f_{1}g_{1}+\kappa_{1}/\nu_{7},g_{1})$.
Then the point $P_{7}:(\kappa_{1}/\nu_{7},0)$ corresponds to the line $g_{1}=0$.
By the transformation $(f,g)=(f_{1}g_{1}+\kappa_{1}/\nu_{7},g_{1})$, the operator $L_{1}$ is written as 
\begin{align}
    L_{1}=&-\frac{\nu_{1}\nu_{2}\nu_{3}\nu_{4}(g_{1}\kappa_{2}-\nu_{5})(g_{1}\kappa_{2}-\nu_{6})\nu_{7}}{g_{1}{\kappa_{2}}^{2}(\kappa_{1}+f_{1}g_{1}\nu_{7})}+\frac{(g_{1}\nu_{1}-1)(g_{1}\nu_{2}-1)z}{g_{1}q} \nonumber \\
    &+\frac{\nu_{1}\nu_{2}\nu_{7}(q\nu_{3}-z)(q\nu_{4}-z)(g_{1}-T_{z}^{-1})}{q(\kappa_{1}q+f_{1}g_{1}q\nu_{7}-\nu_{7}z)}-\frac{(\kappa_{1}-\nu_{7}z)(\kappa_{1}-\nu_{8}z)(g_{1}T_{z}-1)}{g_{1}q\nu_{8}(\kappa_{1}+f_{1}g_{1}\nu_{7}-\nu_{7}z)} . \nonumber 
\end{align}
We consider the equation $L_1 y(z)=0 $.
It follows from the relation $\kappa_{1} ^2 \kappa_{2} ^2 = q \nu_{1} \nu_{2} \dots \nu_{8} $ that divergence as $g_{1} \to 0$ in $L_1 y(z)=0 $ was avoided and we have
\begin{align}
    &-\frac{\kappa_{1}-\nu_{8}z}{q\nu_{8}}y(qz)-\frac{\nu_{1}\nu_{2}\nu_{7}(z-q\nu_{3})(z-q\nu_{4})}{q(\kappa_{1}q-\nu_{7}z)}y(z/q) \nonumber \\
    &+\Bigl\{\frac{\nu_{1}\nu_{2}\nu_{3}\nu_{4}\nu_{7}(\nu_{5} +\nu_{6})}{\kappa_{2}(\kappa_{1}-\nu_{7}z)} -\frac{\kappa_{1}( \nu_{1}+ \nu_{2} ) z}{q(\kappa_{1}-\nu_{7}z)}-\frac{\nu_{1}\nu_{2}\nu_{3}\nu_{4}(\nu_{5}+ \nu_{6} ) {\nu_{7}}^{2}z}{\kappa_{1}\kappa_{2}(\kappa_{1}-\nu_{7}z)} \nonumber \\
    &\qquad +\frac{f_{1}\nu_{7}z}{q(\kappa_{1}-\nu_{7}z)}-\frac{f_{1} {\nu_{7}}^{2}z}{q \nu _8 (\kappa_{1}-\nu_{7}z)}+\frac{(\nu_{1} + \nu_{2} ) \nu_{7}z^{2}}{q(\kappa_{1}-\nu_{7}z)}\Bigr\}y(z)=0 \nonumber 
\end{align}
by setting $g_1=0$.
To obtain the $q$-Heun equation, we transform the dependent variable $y(z)$ by setting $u(z)=y(z)/(\nu_{7}z - \kappa_{1})$.
Then it follows that
\begin{align}
    &(z-q\nu_{3})(z-q\nu_{4})u(z/q)+\frac{q^{2}}{\nu_{1}\nu_{2}}\Bigl(z-\frac{\kappa_{1}}{q\nu_{7}}\Bigr)\Bigl(z-\frac{\kappa_{1}}{\nu_{8}}\Bigr)u(qz) \nonumber \\
    & -\Bigl[\Bigl(\frac{q}{\nu_{1}}+\frac{q}{\nu_{2}}\Bigr)z^{2} +\Bigl\{ \frac{q f_{1}}{\nu_{1}\nu_{2}} \Bigl( 1-\frac{\nu_{7}}{\nu_{8}} \Bigr) -\frac{\kappa_{1}q}{\nu_{7}} \Bigl( \frac{1}{\nu_{1}} + \frac{1}{\nu_{2}} \Bigr) -\frac{q^{2}\nu_{3}\nu_{4} \nu_{7} (\nu_{5} + \nu_{6} ) }{\kappa_{1}\kappa_{2}} \Bigr\} z \nonumber \\
    &\qquad +\frac{q^{3/2}\kappa_{1}{\nu_{3}}^{1/2}{\nu_{4}}^{1/2}}{{\nu_{1}}^{1/2}{\nu_{2}}^{1/2}{\nu_{7}}^{1/2}{\nu_{8}}^{1/2}}\Bigl\{\Bigl(\frac{\nu_{6}}{\nu_{5}}\Bigr)^{1/2}+\Bigl(\frac{\nu_{6}}{\nu_{5}}\Bigr)^{-1/2}\Bigr\}\Big]u(z)=0 . \nonumber 
\end{align}
We compare it with the $q$-Heun equation given in Eq.~(\ref{eq:qHeunD15}).
Then we obtain the following correspondence
\begin{align}
    &h_{1}=q^{1/2}\nu_{3},\quad h_{2}=q^{1/2}\nu_{4},\quad h_{3}=\frac{\nu_{6}}{\nu_{5}} \nonumber \\
    &l_{1}=\frac{\kappa_{1}}{q^{1/2}\nu_{7}},\quad l_{2}=\frac{q^{1/2}\kappa_{1}}{\nu_{8}},\quad l_{3}=\frac{q}{\nu_{1}},\quad l_{4}=\frac{q}{\nu_{2}}. \nonumber 
\end{align}
We may regard $f_1$ as an accessory parameter.

On the blow-up at the point $P_{8}:(\kappa_{1}/\nu_{8},0)$, we obtain the $q$-Heun equation where the parameters $\nu _7$ and $\nu _8$ are exchanged.

\subsubsection{The points $P_{1}:(\infty,1/\nu_{1})$ and $P_{2}:(\infty,1/\nu_{2})$} $ $

We consider the blow-up at the point $P_1: (\infty,1/\nu_{1})$.
Set $(f,g)=(1/f_{0},g_{0})$. Then the point $(f,g) = (\infty,1/\nu_{1})$ corresponds to the point $(f_{0},g_{0})=(0 ,1/\nu_{1})$.
We realize the blow-up at $(f_{0},g_{0})=(0 ,1/\nu_{1})$ by setting $(f_{0},g_{0})=(f_{1},f_{1}g_{1}+1/\nu_{1})$.
Then the point $P_{1}:(\infty,\frac{1}{\nu_{1}})$ corresponds to the line $f_{1}=0$.
We substitute $(f,g)=(1/f_{0},g_{0})$ and $(f_{0},g_{0})=(f_{1},f_{1}g_{1}+1/\nu_{1})$ into the equation $L_1 y(z)=0 $.
We set $f_1=0$.
By applying the relation $\kappa_{1} ^2 \kappa_{2} ^2 = q \nu_{1} \nu_{2} \dots \nu_{8} $, we obtain 
\begin{align}
    &(z-q\nu_{3})(z-q\nu_{4})y(z/q)+\frac{q}{\nu_{1}\nu_{2}}\Bigl(z-\frac{\kappa_{1}}{\nu_{7}}\Bigr)\Bigl(z-\frac{\kappa_{1}}{\nu_{8}}\Bigr)y(qz) \nonumber \\
    &-\Bigl[\Bigl(\frac{1}{\nu_{1}}+\frac{q}{\nu_{2}}\Bigr)z^{2}+\Bigl\{ g_{1}q \Bigl( 1 -\frac{\nu_{1}}{\nu_{2}}\Bigr) - \frac{q}{\nu_{1}} ( \nu_{3} + \nu_{4} ) - \frac{\kappa_{1}q}{\nu_{2}} \Bigl( \frac{1}{\nu_{7}} + \frac{1}{\nu_{8}} \Bigr) \Bigr\} z \nonumber \\
& \qquad +\Bigl(\frac{q^{3}\nu_{3}\nu_{4}{\kappa_{1}}^{2}}{\nu_{1}\nu_{2}\nu_{7}\nu_{8}}\Bigr)^{1/2}\Bigl\{\Bigl(\frac{\nu_{6}}{\nu_{5}}\Bigr)^{1/2}+\Bigl(\frac{\nu_{6}}{\nu_{5}}\Bigr)^{-1/2}\Bigr\}\Bigr]y(z)=0 .\nonumber 
\end{align}
We compare it with the $q$-Heun equation given in Eq.~(\ref{eq:qHeunD15}).
Then we obtain the following correspondence
\begin{align}
    &h_{1}=q^{1/2}\nu_{3},\quad h_{2}=q^{1/2}\nu_{4},\quad h_{3}=\frac{\nu_{6}}{\nu_{5}} \nonumber \\
    &l_{1}=\frac{q^{1/2}\kappa_{1}}{\nu_{7}},\quad l_{2}=\frac{q^{1/2}\kappa_{1}}{\nu_{8}},\quad l_{3}=\frac{1}{\nu_{1}},\quad l_{4}=\frac{q}{\nu_{2}}. \nonumber 
\end{align}
We may regard $g_1$ as an accessory parameter.

On the blow-up at the point $P_2: (\infty,1/\nu_{2})$, we obtain the $q$-Heun equation where the parameters $\nu _1$ and $\nu _2$ are exchanged.

\subsubsection{The points $P_{3}:(\nu_{3},\infty)$ and $P_{4}:(\nu_{4},\infty)$} $ $

We consider the blow-up at the point $P_{3}:(\nu_{3},\infty)$.
Set $(f,g)=(f_{0},1/g_{0})$. Then the point $(f,g) = (\nu_{3},\infty)$ corresponds to the point $(f_{0},g_{0})=(\nu_{3},0)$.
We realize the blow-up at $(f_{0},g_{0})=(\nu_{3},0)$ by setting $(f_{0},g_{0})=(f_{1}g_{1}+\nu_{3},g_{1})$.
Then the point $P_{3}:(\nu_{3},\infty)$ corresponds to the line $g_{1}=0$.
We substitute $(f,g)=(f_{0},1/g_{0})$ and $(f_{0},g_{0})=(f_{1}g_{1}+\nu_{3},g_{1})$ into the equation $L_1 y(z)=0 $.
Set $g_{1}=0$.
Then we have
\begin{align}
    &-\frac{(\kappa_{1}-\nu_{7}z)(\kappa_{1}-\nu_{8}z)}{q\nu_{7}\nu_{8}(\nu_{3}-z)}y(qz)-\frac{\nu_{1}\nu_{2}(q\nu_{4}-z)}{q}y(z/q) \nonumber \\
    &+\frac{1}{\kappa_{2}q (q\nu_{3}-z)}\Bigl[ q^{2}\nu_{1}\nu_{2}{\nu_{3}}\nu_{4}(\nu_{5} +\nu_{6}) +\{ q f_{1}\kappa_{2} \nu_{1}\nu_{2} (1 - \nu_{4} / \nu _3 )  \nonumber \\
    & \qquad -q \kappa_{2} {\nu_{3}} (\nu_{1} +\nu_{2}) -q\nu_{1}\nu_{2}\nu_{4}(\nu_{5} + \nu_{6}) \} z+ \kappa_{2} (\nu_{1}+\nu_{2})z^{2}\Bigr] y(z)=0. \nonumber
\end{align}
To obtain the $q$-Heun equation, we set $u(z)=y(z)/(z- q\nu_{3})$.
Then 
\begin{align}
    &(z-q^{2}\nu_{3})(z-q\nu_{4})u(z/q)+\frac{q^{2}}{\nu_{1}\nu_{2}}\Bigl(z-\frac{\kappa_{1}}{\nu_{7}}\Bigr)\Bigl(z-\frac{\kappa_{1}}{\nu_{8}}\Bigr)u(qz) \nonumber \\
    &-\Bigl[q \Bigl( \frac{1}{\nu_{1}}+\frac{1}{\nu_{2}} \Bigr) z^{2}+ q^2 \Bigl\{ f_{1}\Bigl( 1 -\frac{\nu_{4}}{\nu_{3}} \Bigr) - \nu_{3} \Bigl( \frac{1}{\nu_{1}} + \frac{1}{\nu_{2}} \Bigr) - \frac{\nu_{4}}{\kappa_{2}} ( \nu_{5} + \nu_{6} ) \Bigr\} z \nonumber \\
& \qquad +\frac{q^{3}\nu_{3}\nu_{4}{\nu_{5}}^{1/2}{\nu_{6}}^{1/2}}{\kappa_{2}}\Bigl\{\Bigl(\frac{\nu_{6}}{\nu_{5}}\Bigr)^{1/2}+\Bigl(\frac{\nu_{6}}{\nu_{5}}\Bigr)^{-1/2}\Bigr\}
    \Bigr] u(z) =0 . \nonumber
\end{align}
We compare it with the $q$-Heun equation given in Eq.~(\ref{eq:qHeunD15}).
Then we obtain the following correspondence
\begin{align}
&h_{1}=q^{3/2}\nu_{3}, \quad h_{2}=q^{1/2}\nu_{4}, \quad h_{3}=\frac{\nu_{6}}{\nu_{5}}, \nonumber \\
&l_{1}=\frac{q^{1/2}\kappa_{1}}{\nu_{7}},\quad l_{2}=\frac{q^{1/2}\kappa_{1}}{\nu_{8}},\quad l_{3}=\frac{q}{\nu_{1}},\quad l_{4}=\frac{q}{\nu_{2}}. \nonumber 
\end{align}
We may regard $f_1$ as an accessory parameter.

On the blow-up at the point $P_4:(\nu_{4},\infty)$, we obtain the $q$-Heun equation where the parameters $\nu _3$ and $\nu _4$ are exchanged.

\subsubsection{The other cases related to the $q$-Heun equation} $ $

We can obtain the $q$-Heun equation by other restrictions of the parameters.
We substitute $f= \kappa_{1}/\nu_{7} $ to the equation $L_1 y(z)=0 $.
Then we have
\begin{align}
    &(z-q\nu_{3})(z-q\nu_{4})y(z/q)+\frac{1}{\nu_{1}\nu_{2}}\Bigl(z-\frac{q\kappa_{1}}{\nu_{7}}\Bigr)\Bigl(z-\frac{\kappa_{1}}{\nu_{8}}\Bigr)y(qz) -\Bigl[\Bigl(\frac{1}{\nu_{1}}+\frac{1}{\nu_{2}}\Bigr)z^{2} \nonumber \\
    & +\Bigl\{ \frac{q (\nu _3 \nu_7 - \kappa_{1} )(\nu _4 \nu_7 - \kappa_{1} )}{\nu_{7} \kappa_{1}}g -\frac{\kappa_{1}q}{\nu_{7}} \Bigl( \frac{1}{\nu_{1}} + \frac{1}{\nu_{2}} \Bigr) -\frac{q \nu_{3}\nu_{4} \nu_{7} (\nu_{5} + \nu_{6} ) }{\kappa_{1}\kappa_{2}} \Bigr\} z \nonumber \\
    &\qquad +\frac{q^{3/2}\kappa_{1}{\nu_{3}}^{1/2}{\nu_{4}}^{1/2}}{{\nu_{1}}^{1/2}{\nu_{2}}^{1/2}{\nu_{7}}^{1/2}{\nu_{8}}^{1/2}}\Bigl\{\Bigl(\frac{\nu_{6}}{\nu_{5}}\Bigr)^{1/2}+\Bigl(\frac{\nu_{6}}{\nu_{5}}\Bigr)^{-1/2}\Bigr\}\Big] y(z)=0 . \nonumber 
\end{align}
Hence we obtain the $q$-Heun equation, and we may regard $g$ as an accessory parameter.
We have a similar result by substituting $f= \kappa_{1}/\nu_{8} $ to the equation $L_1 y(z)=0 $.

We substitute $f= \nu_{3} $ to the equation $L_1 y(z)=0 $.
Then we have
\begin{align}
    &(z-\nu_{3})(z-q\nu_{4})y(z/q)+\frac{1}{\nu_{1}\nu_{2}}\Bigl(z-\frac{\kappa_{1}}{\nu_{7}}\Bigr)\Bigl(z-\frac{\kappa_{1}}{\nu_{8}}\Bigr)y(qz) \nonumber \\
    &-\Bigl[\Bigl(\frac{1}{\nu_{1}}+\frac{1}{\nu_{2}}\Bigr)z^{2} - \Bigl( \frac{ (\nu _3 \nu_7 - \kappa_{1} )(\nu _4 \nu_7 - \kappa_{1} )}{\nu_{1} \nu_{2} \nu_{3} \nu_{7} \nu_{7} }\frac{1}{g} +\frac{\nu_{3}}{\nu_{1}} + \frac{\nu_{3}}{\nu_{2}}+\frac{q\nu_{4}\nu_{5}}{\kappa_{2}}+\frac{q\nu_{4}\nu_{6}}{\kappa_{2}}\Bigr)z \nonumber \\
& \qquad +\frac{q \nu_{3}\nu_{4}{\nu_{5}}^{1/2}{\nu_{6}}^{1/2}}{\kappa_{2}}\Bigl\{\Bigl(\frac{\nu_{6}}{\nu_{5}}\Bigr)^{1/2}+\Bigl(\frac{\nu_{6}}{\nu_{5}}\Bigr)^{-1/2}\Bigr\}
    \Bigr] y(z) =0 . \nonumber
\end{align}
Hence we obtain the $q$-Heun equation, and we may regard $g$ as an accessory parameter.
Note that this procedure was essentially considered in section 3.1 of \cite{TakR}.
We have a similar result by substituting $f= \nu_{4} $ to the equation $L_1 y(z)=0 $.

\subsection{The case $E^{(1)}_{6}$} $ $

The $q$-Painlev\'e equation of type $E^{(1)}_{6} $ ($q$-$P(E^{(1)}_{6})$) in \cite{KNY} was given as
\begin{align}
 & \frac{(fg-1)(\overline{f}g-1)}{f\overline{f}}=\frac{(g- 1 / \nu_{1} ) (g- 1 / \nu_{2} ) (g- 1 / \nu_{3} ) (g- 1 / \nu_{4} )}{ (g- \nu_{5}/\kappa_{2} ) (g- \nu_{6}/\kappa_{2} )}, \nonumber \\
 & \frac{(fg-1)(f\underline{g}-1)}{g\underline{g}}=\frac{(f-\nu_{1}) (f-\nu_{2}) (f-\nu_{3}) (f-\nu_{4})}{ (f- \kappa_{1}/\nu_{7} )(f- \kappa_{1}/\nu_{8} )} . \nonumber
\end{align}
The space of initial condition for $q$-$P(E^{(1)}_{6})$ was realized by blowing up eight points $P_{1}, \dots , P_8 $ of ${\mathbb P^1}\times {\mathbb P^1}$, where 
\begin{equation}
P_{i}:\Bigl(\nu_{i},\frac{1}{\nu_{i}}\Bigr)_{i=1,2,3,4},\quad \Bigl(0,\frac{\nu_{i}}{\kappa_{2}}\Bigr)_{i=5,6},\quad \Bigl(\frac{\kappa_{1}}{\nu_{i}},0\Bigr)_{i=7,8} \, .
\nonumber 
\end{equation}
On the other hand, $q$-$P(E^{(1)}_{6})$ is obtained by the compatibility condition for the Lax operators $L_1$ and $L_2$, where
\begin{align}
    L_{1}=&\frac{ z (g\nu_{1}-1)(g\nu_{2}-1)(g\nu_{3}-1)(g\nu_{4}-1)}{g(fg-1) ( g z - q )}-\frac{( g\kappa_{2}/\nu_{5}-1 )( g\kappa_{2}/\nu_{6}-1 ){\kappa_{1}}^{2}}{q fg \nu_{7}\nu_{8}} \nonumber \\
& +\frac{(\nu_{1}- z/q )(\nu_{2}- z/q )(\nu_{3}- z/q )(\nu_{4}- z/q )}{f- z/q}\Bigl\{\frac{g}{1-g z/q}-T_{z}^{-1}\Bigr\} \nonumber \\
    &+\frac{( \kappa_{1}/\nu_{7}-z )( \kappa_{1}/\nu_{8}-z )}{q(f-z)}\Bigl\{\Bigl(\frac{1}{g}-z\Bigr)-T_{z}\Bigr\} \, , \nonumber  \\
    L_{2}=&\Bigl(1-\frac{f}{z}\Bigr)T+T_{z}-\Bigl(\frac{1}{g}-z\Bigr) \, . \nonumber 
\end{align}
We track the linear differential equation $L_1 y(z)= 0$ on the process of the blow-up of the point $P_i$ $(i=1,2,\dots ,8)$ respectively, and we investigate a relationship with the variant of the $q$-Heun equation of degree $3$;
\begin{align}
&(x-h_{1}{q}^{1/2})(x-h_{2}{q}^{1/2})(x-h_{3}{q}^{1/2})g(x/q)\label{eq:qHeunE16} \\
& +(x-l_{1}{q}^{-1/2})(x-l_{2}{q}^{-1/2})(x-l_{3}{q}^{-1/2})g(qx) \nonumber \\
&+\{-({q}^{1/2}+{q}^{-1/2})x^{3}+(h_{1}+h_{2}+h_{3}+l_{1}+l_{2}+l_{3})x^{2} \nonumber \\
&\qquad \qquad -Ex+(l_{1}l_{2}l_{3}h_{1}h_{2}h_{3})^{1/2}(h_{4}^{-1/2}+h_{4}^{1/2})\}g(x)=0 \, . \nonumber
\end{align}
Note that detailed calculation in the following subsections was performed in \cite{Ssk}.

\subsubsection{The points $P_{5}:(0, \nu_{5}/\kappa_{2})$ and $P_{6}:(0,\nu_{6}/\kappa_{2})$} $ $

We realize the blow-up at $(f,g)= (0, \nu_{5}/\kappa_{2})$ by setting $(f,g)=(f_{1}g_{1},g_{1}+\nu_{5}/\kappa_{2})$.
Then the point $P_{5}:(0, \nu_{5}/\kappa_{2})$ corresponds to the line $g_{1}=0$.
We substitute $(f,g)=(f_{1}g_{1},g_{1}+\nu_{5}/\kappa_{2})$ into the equation $L_1 y(z)=0 $, and set $g_{1}=0$.
Then we have
\begin{align}
&\Bigl(\frac{\kappa_{1}}{\nu_{7}}-z\Bigr)\Bigl(\frac{\kappa_{1}}{\nu_{8}}-z\Bigr)y(qz)+\frac{(q\nu_{1}-z)(q\nu_{2}-z)(q\nu_{3}-z)(q\nu_{4}-z)}{q^{2}}y(z/q) \label{eq:E6P5y} \\
&+\Bigl\{-\frac{{\kappa_{1}}^{2}{\kappa_{2}}^{2}(\nu_{5}-\nu_{6})}{f_{1}{\nu_{5}}^{2}\nu_{6}\nu_{7}\nu_{8}}z+\Bigl(\frac{\kappa_{1}}{\nu_{7}}-z\Bigr)\Bigl(\frac{\kappa_{1}}{\nu_{8}}-z\Bigr)z-\frac{\kappa_{2}(\frac{\kappa_{1}}{\nu_{7}}-z)(\frac{\kappa_{1}}{\nu_{8}}-z)}{\nu_{5}} \nonumber \\
&\qquad +\frac{\nu_{5}(q\nu_{1}-z)(q\nu_{2}-z)(q\nu_{3}-z)(q\nu_{4}-z)}{q(-\kappa_{2}q+\nu_{5}z)} \nonumber \\
&\qquad -\frac{(-\kappa_{2}+\nu_{1}\nu_{5})(-\kappa_{2}+\nu_{2}\nu_{5})(-\kappa_{2}+\nu_{3}\nu_{5})(-\kappa_{2}+\nu_{4}\nu_{5})qz^{2}}{{\kappa_{2}}^{2}\nu_{5}(-\kappa_{2}q+\nu_{5}z)}\Bigr\}y(z)=0 \, . \nonumber 
\end{align}
To obtain a variant of the $q$-Heun equation, we apply gauge transformations.
Set
\begin{equation}
(\alpha x;q)_{\infty}=\prod_{i=0}^{\infty} (1-q^{i}\alpha x)=(1-\alpha x)(1-q\alpha x)(1-q^{2}\alpha x)\cdots \nonumber
\end{equation}
\begin{prop} \label{prop:GT} $ $\\
(i) If $y(x)$ satisfies $q^{-\lambda}a(x)g(x/q)+b(x)g(x)+q^{\lambda}c(x)g(qx)=0$, then the function $h(x)=x^{\lambda}y(x)$ satisfies $a(x)g(x/q)+b(x)g(x)+c(x)g(qx)=0$.\\
(ii) If $y(x)$ satisfies $(1-\alpha x)a(x)g(x/q)+b(x)g(x)+c(x)g(qx)=0$, then the function $u(x)=(\alpha qx;q)_{\infty}y(x)$ satisfies $a(x)g(x/q)+b(x)g(x)+(1-\alpha qx)c(x)g(qx)=0$.
\end{prop}
\begin{proof}
If $h(x)=x^{\lambda}y(x)$, then $y(qx) = q^{-\lambda} x^{-\lambda} h(qx)$ and $y(x/q) = q^{\lambda} x^{-\lambda} h(x/q)$.
Hence the equation $q^{-\lambda}a(x) h(x/q)+b(x)g(x)+q^{\lambda}c(x) h (qx)=0$ is rewritten as $a(x)y(x/q)+b(x)y(x)+c(x)y(qx)=0$ and we obtain (i).

We can obtain (ii) similarly, as shown in \cite{MST}.
\end{proof}
We apply a gauge transformation in Proposition \ref{prop:GT} (ii) to Eq.~(\ref{eq:E6P5y}).
Then we obtain 
\begin{align}
&\frac{q^{1/2}}{\nu_{1}}(\nu_{1}-z)\Bigl(\frac{\kappa_{1}}{\nu_{7}}-z\Bigr)\Bigl(\frac{\kappa_{1}}{\nu_{8}}-z\Bigr)\tilde{y}(qz)+\frac{\nu_{1}}{q^{1/2}}(q\nu_{2}-z)(q\nu_{3}-z)(q\nu_{4}-z)\tilde{y}(z/q) \nonumber \\
&+q^{1/2} \Bigl\{-\frac{q^{2}\nu_{1}\nu_{2}\nu_{3}\nu_{4}\nu_{5}}{\kappa_{2}}-\frac{{\kappa_{1}}^{2}\kappa_{2}}{\nu_{5}\nu_{7}\nu_{8}}+\Bigl(\frac{q\nu_{1}\nu_{2}\nu_{3}\nu_{5}}{\kappa_{2}}+\frac{q\nu_{1}\nu_{2}\nu_{4}\nu_{5}}{\kappa_{2}}+\frac{q\nu_{1}\nu_{3}\nu_{4}\nu_{5}}{\kappa_{2}} \nonumber \\
& +\frac{q\nu_{2}\nu_{3}\nu_{4}\nu_{5}}{\kappa_{2}} -\frac{q\nu_{1}\nu_{2}\nu_{3}\nu_{4}{\nu_{5}}^{2}}{{\kappa_{2}}^{2}}+\frac{\kappa_{1}\kappa_{2}}{\nu_{5}\nu_{7}}+\frac{\kappa_{1}\kappa_{2}}{\nu_{5}\nu_{8}}+\frac{{\kappa_{1}}^{2}}{\nu_{7}\nu_{8}}+\frac{{\kappa_{1}}^{2}{\kappa_{2}}^{2}}{f_{1}{\nu_{5}}^{2}\nu_{7}\nu_{8}}-\frac{{\kappa_{1}}^{2}{\kappa_{2}}^{2}}{f_{1}\nu_{5}\nu_{6}\nu_{7}\nu_{8}}\Bigr)z \nonumber \\
&-\Bigl(\nu_{1}+\nu_{2}+\nu_{3}+\nu_{4} +\frac{\kappa_{1}}{\nu_{7}}+\frac{\kappa_{1}}{\nu_{8}}\Bigr)z^{2}+\Bigl(1+\frac{1}{q}\Bigr)z^{3}\Bigr\}\tilde{y}(z)=0 \, . \nonumber 
\end{align}
By applying the relation $\kappa_{1} ^2 \kappa_{2} ^2 = q \nu_{1} \nu_{2} \dots \nu_{8} $ and a gauge transformation in Proposition \ref{prop:GT} (i), we have
\begin{align}
  &(z-q\nu_{2})(z-q\nu_{3})(z-q\nu_{4}) u(z/q)+(z-\nu_{1})\Bigl(z-\frac{\kappa_{1}}{\nu_{7}}\Bigr)\Bigl(z-\frac{\kappa_{1}}{\nu_{8}}\Bigr) u(qz) \label{eq:E6P5u} \\
  &+\Bigl[-(q^{1/2}+q^{-1/2})z^{3} +q^{1/2} \Bigl(\nu_{1}+\nu_{2}+\nu_{3}+\nu_{4}+\frac{\kappa_{1}}{\nu_{7}}+\frac{\kappa_{1}}{\nu_{8}}\Bigr)z^{2} \nonumber \\
  & \qquad -q^{3/2} \Bigl\{ \nu_1 \nu_2 \nu_3 \nu_4 \Bigl(  \frac{ \nu_6 }{ \nu_5} - 1 \Bigr) \Bigl( \frac{1}{f_{1}} + \frac{{\nu_{5}}^{2}}{{\kappa_{2}}^{2}} \Bigr) + \nu_{1}\nu_{2}\nu_{3}\nu_{4}\frac{\nu_{5}}{\kappa_{2}}\Bigl( \frac{1}{\nu_1} +\frac{1}{\nu_2} +\frac{1}{\nu_3} +\frac{1}{\nu_4} \Bigr) \nonumber \\
  & \qquad +\frac{\kappa_{2}}{q\nu_{5}} \Bigl( \frac{\kappa_{1}}{\nu_{7}} +\frac{\kappa_{1}}{\nu_{8}}\Bigr) \Bigr\} z +\Bigl(q^{3}\frac{\nu_{1}\nu_{2}\nu_{3}\nu_{4}{\kappa_{1}}^{2}}{\nu_{7}\nu_{8}}\Bigr)^{1/2}\Bigl\{\Bigl(\frac{q\nu_{5}}{\nu_{6}}\Bigr)^{1/2}+\Bigl(\frac{q\nu_{5}}{\nu_{6}}\Bigr)^{-1/2}\Bigr\}\Bigr] u(z)=0 \, . \nonumber 
\end{align}
In fact, if $y(z)$ is a solution of Eq.~(\ref{eq:E6P5y}) and the parameter $\lambda $ satisfies $q^{\lambda } = q^{1/2}/ \nu _1 $, then the function $u(z)= z^{\lambda} (z/\nu _1; q)_{\infty } y(z)$ satisfies Eq.~(\ref{eq:E6P5u}).
We compare Eq.~(\ref{eq:E6P5u}) with the variant of the $q$-Heun equation of degree $3$ given in Eq.~(\ref{eq:qHeunE16}).
Then we obtain the following correspondence
\begin{align}
&h_{1}=q^{1/2}\nu_{2},\quad h_{2}=q^{1/2}\nu_{3},\quad h_{3}=q^{1/2}\nu_{4},\quad h_{4}=\frac{q\nu_{5}}{\nu_{6}}, \nonumber \\
&l_{1}=q^{1/2}\nu_{1},\quad l_{2}=\frac{q^{1/2}\kappa_{1}}{\nu_{7}},\quad l_{3}=\frac{q^{1/2}\kappa_{1}}{\nu_{8}} .\nonumber 
\end{align}
We may regard $f_1$ as an accessory parameter.

On the blow-up at the point $P_{6}$, we obtain the equation where the parameters $\nu _5$ and $\nu _6$ are exchanged.

\subsubsection{The points $P_{7}:(\kappa_{1}/\nu_{7},0)$ and $P_{8}:(\kappa_{1}/\nu_{8},0)$} $ $

We realize the blow-up at $(f,g)= (\kappa_{1}/\nu_{7},0)$ by setting $(f,g)=(f_{1}g_{1}+\kappa_{1}/\nu_{7},g_{1})$.
Then the point $P_{7}:(\kappa_{1}/\nu_{7},0)$ corresponds to the line $g_{1}=0$.
We substitute $(f,g)=(f_{1}g_{1}+\kappa_{1}/\nu_{7},g_{1})$ into the equation $L_1 y(z)=0 $, and set $g_{1}=0$.
Then we have
\begin{align}
 & \frac{z-\kappa_{1}/\nu_{8}}{q} y(qz) +\frac{(q\nu_{1}-z)(q\nu_{2}-z)(q\nu_{3}-z)(q\nu_{4}-z)}{q^{3}(z -q \kappa_{1}/ \nu_{7} )} y(z/q) \nonumber \\
  &+\Bigl[ \frac{{\kappa_{1}}^{2}q\{ - \kappa_{2}(\nu_{5}+\nu_{6}) + \nu_{5}\nu_{6}(1 -\nu_{8} /\nu_{7} )z \}}{q^{2}\nu_{5}\nu_{6}\nu_{7}\nu_{8}(z -\kappa_{1}/\nu _7 )} \nonumber \\
  &+\frac{z\{f_{1}q(\nu_{7}-\nu_{8})+\nu_{8}z( (q +1 ) z-q(\nu_{1}+\nu_{2}+\nu_{3}+\nu_{4}) )\}}{q^{2}\nu_{8}( z -\kappa_{1}/{\nu_{7}} )} \nonumber \\
  &+\frac{\kappa_{1}z\{\kappa_{2}q(\nu_{5}+\nu_{6})\nu_{7}+\nu_{5}\nu_{6}(-(q \nu_{7}+ \nu_{8})z+q(\nu_{1}+\nu_{2}+\nu_{3}+\nu_{4}) \nu_{8} )\}}{q^{2}\nu_{5}\nu_{6}\nu_{7}\nu_{8}(z -\kappa_{1}/\nu_{7})} \Bigr] y(z) =0 \, . \nonumber 
\end{align}
By setting $\tilde{y}(z)=y(z)/(z-\kappa_{1}/\nu_{7})$, we have
\begin{align}
&\frac{(z-\kappa_{1}/\nu_{8})(qz-\kappa_{1}/\nu_{7})}{q}\tilde{y}(qz)+\frac{(q\nu_{1}-z)(q\nu_{2}-z)(q\nu_{3}-z)(q\nu_{4}-z)}{q^{4}}\tilde{y}(z/q) \nonumber \\
  &+\Bigl[ \frac{{\kappa_{1}}^{2}q\{ - \kappa_{2}(\nu_{5}+\nu_{6}) + \nu_{5}\nu_{6}(1 -\nu_{8} /\nu_{7} )z \}}{q^{2}\nu_{5}\nu_{6}\nu_{7}\nu_{8}} \nonumber \\
  &+\frac{z\{f_{1}q(\nu_{7}-\nu_{8})+\nu_{8}z( (q +1 ) z-q(\nu_{1}+\nu_{2}+\nu_{3}+\nu_{4}) )\}}{q^{2}\nu_{8}} \nonumber \\
  &+\frac{\kappa_{1}z\{\kappa_{2}q(\nu_{5}+\nu_{6})\nu_{7}+\nu_{5}\nu_{6}(-(q \nu_{7}+ \nu_{8})z+q(\nu_{1}+\nu_{2}+\nu_{3}+\nu_{4}) \nu_{8} )\}}{q^{2}\nu_{5}\nu_{6}\nu_{7}\nu_{8}} \Bigr] \tilde{y}(z)=0 \, .\nonumber 
\end{align}
Let $\lambda $ be the parameter such that $q^{\lambda } = q^{3/2}/ \nu _1 $, then the function $u(z)= z^{\lambda} (z/\nu _1; q)_{\infty } \tilde{y}(z)$ satisfies 
\begin{align}
&(z-q\nu_{2})(z-q\nu_{3})(z-q\nu_{4})u(z/q)+(z-\nu_{1})\Bigl(z-\frac{\kappa_{1}}{\nu_{7}q}\Bigr)\Bigl(z-\frac{\kappa_{1}}{\nu_{8}}\Bigr)u(qz) \nonumber \\
&+\Bigl[-(q^{1/2}+q^{-1/2})z^{3} + q^{1/2}\Bigl(\nu_{1}+\nu_{2}+\nu_{3}+\nu_{4}+\frac{\kappa_{1}}{\nu_{7}q}+\frac{\kappa_{1}}{\nu_{8}}\Bigr)z^{2}\nonumber \\
& -q^{1/2}\Bigl\{ \Bigl( f_{1} + \frac{{\kappa_{1}}^{2}}{{\nu_{7}}^{2}} \Bigr)  \Bigl( \frac{\nu_{7}}{\nu_{8}} -1  \Bigr)  +\frac{\kappa_{1}}{\nu_{7}} ( \nu_{1} +  \nu_{2} +  \nu_{3} +  \nu_{4} ) +\frac{\kappa_{1}}{\nu_{8}} \Bigl( \frac{\kappa_{2}}{\nu_{5}} +\frac{\kappa_{2}}{\nu_{6} } \Bigr) \Bigr\} z \nonumber \\
& +\Bigl(\frac{q^{2}{\kappa_{1}}^{2}\nu_{1}\nu_{2}\nu_{3}\nu_{4}}{\nu_{7}\nu_{8}}\Bigr)^{1/2}\Bigl\{\Bigl(\frac{\nu_{5}}{\nu_{6}}\Bigr)^{1/2} +\Bigl(\frac{\nu_{5}}{\nu_{6}}\Bigr)^{-1/2}\Bigr\}\Bigr]u(z)=0 \, .\nonumber 
\end{align}
We compare it with the variant of the $q$-Heun equation of degree $3$ given in Eq.~(\ref{eq:qHeunE16}).
Then we obtain the following correspondence
\begin{align}
& h_{1}=q^{1/2}\nu_{1},\quad h_{2}=q^{1/2}\nu_{3},\quad h_{3}=q^{1/2}\nu_{4},\quad h_{4}=\frac{\nu_{5}}{\nu_{6}}, \nonumber \\
&l_{1}=q^{1/2}\nu_{1},\quad l_{2}=\frac{\kappa_{1}}{q^{1/2}\nu_{7}},\quad l_{3}=\frac{q^{1/2}\kappa_{1}}{\nu_{8}} . \nonumber 
\end{align}
We may regard $f_1$ as an accessory parameter.

On the blow-up at the point $P_{8}$, we obtain the equation where the parameters $\nu _7$ and $\nu _8$ are exchanged.

\subsubsection{The points $P_{i}:(\nu_{i},1/\nu_{i})$ $(i =1,2,3,4 )$} $ $

We realize the blow-up at $(f,g)= (\nu_{1},1/\nu_{1})$ by setting $(f,g)=(f_{1}g_{1}+\nu_{1},g_{1}+1/\nu_{1})$.
Then the point $P_{1}:(\nu_{1},1/\nu_{1})$ corresponds to the line $g_{1}=0$.
We substitute $(f,g)=(f_{1}g_{1}+\nu_{1},g_{1}+1/\nu_{1})$ into the equation $L_1 y(z)=0 $, and set $g_{1}=0$.
Then we have
\begin{align}
&\frac{(z- \kappa_{1}/\nu_{7})(z-\kappa_{1}/\nu_{8})}{q(z -\nu_{1})}y(qz)+\frac{(z-q\nu_{2})(z-q\nu_{3})(z- q\nu_{4})}{q^{3}}y(z/q) \nonumber \\
& +\Bigl\{-\frac{{\kappa_{1}}^{2}(1-\kappa_{2}/(\nu_{1}\nu_{5}))(1-\kappa_{2}/(\nu_{1}\nu_{6}))}{q\nu_{7}\nu_{8}} +\frac{(z- \kappa_{1}/\nu_{7})(z- \kappa_{1}/\nu_{8} )}{q} \nonumber \\
& +\frac{\nu_{1}(\nu_{2}-\nu_{1})(\nu_{3}-\nu_{1})(\nu_{4}-\nu_{1})z}{(f_{1}+{\nu_{1}}^{2})(z-q\nu_{1})} -\frac{q(\nu_{2}- z/q )(\nu_{3}- z/q )(\nu_{4}- z/q )}{z-q\nu_{1}}\Bigr\}y(z)=0 \, . \nonumber 
\end{align}
Let $\lambda $ be the parameter such that $q^{\lambda } = q^{1/2}/ \nu _1 $, then the function $u(z)= z^{\lambda} (z/\nu _1; q)_{\infty } y (z)$ satisfies 
\begin{align}
&(z-q\nu_{2})(z-q\nu_{3})(z-q\nu_{4})u(z/q)+(z-q\nu_{1})\Bigl(z-\frac{\kappa_{1}}{\nu_{7}}\Bigr)\Bigl(z-\frac{\kappa_{1}}{\nu_{8}}\Bigr)u(qz) \nonumber  \\
&+\Bigl[-(q^{-1/2}+q^{1/2})z^{3} +q^{1/2} \Bigl(q \nu_{1}+ \nu_{2}+ \nu_{3}+ \nu_{4}+\frac{\kappa_{1}}{\nu_{7}}+\frac{\kappa_{1}}{\nu_{8}}\Bigr)z^{2}\nonumber \\
&\quad -q^{3/2}\Bigl\{ \frac{\nu_{1} (\nu _2 -\nu _1) (\nu _3 -\nu _1) (\nu _4 -\nu _1)} {f_{1}+{\nu_{1}}^{2}} +\nu _1 \Bigl( \frac{\kappa_{1} }{\nu _7} +  \frac{\kappa_{1} }{\nu _8} \Bigr) \nonumber \\
& \qquad \qquad + \nu_{2}\nu_{3}\nu_4 \Bigl( - \frac{1}{\nu _1} + \frac{1}{\nu _2} + \frac{1}{\nu _3} + \frac{1}{\nu _4} + \frac{\nu _5}{\kappa _2}+ \frac{\nu _6}{\kappa _2} \Bigr)  \Bigr\} z \nonumber \\
&\quad +\Bigl(q^{4}\frac{\nu_{1}\nu_{2}\nu_{3}\nu_{4}{\kappa_{1}}^{2}}{\nu_{7}\nu_{8}}\Bigr)^{1/2}\Bigl\{\Bigl(\frac{\nu_{5}}{\nu_{6}}\Bigr)^{1/2}+\Bigl(\frac{\nu_{5}}{\nu_{6}}\Bigr)^{-1/2}\Bigr\}\Bigr]u(z)=0 \, . \nonumber
\end{align}
We compare it with the variant of the $q$-Heun equation of degree $3$ given in Eq.~(\ref{eq:qHeunE16}).
Then we obtain the following correspondence
\begin{align}
&h_{1}=q^{1/2}\nu_{2},\quad h_{2}=q^{1/2}\nu_{3}\quad h_{3}=q^{1/2}\nu_{4},\quad h_{4}=\frac{\nu_{5}}{\nu_{6}}, \nonumber \\
& l_{1}=q^{3/2}\nu_{1},\quad l_{2}=\frac{q^{1/2}\kappa_{1}}{\nu_{7}},\quad l_{3}=\frac{q^{1/2}\kappa_{1}}{\nu_{8}} . \nonumber 
\end{align}
We may regard $f_1$ as an accessory parameter.

On the blow-up at the point $P_{i}$ $(i=2,\, 3,\, 4)$, we obtain the equation where the parameters $\nu _1$ and $\nu _i$ are exchanged.

\subsubsection{The other cases related to the variant of the $q$-Heun equation of degree three} $ $

We substitute $f= \kappa_{1}/\nu_{7} $ to the equation $L_1 y(z)=0 $.
Then we have
\begin{align}
& (z-q\kappa_{1}/\nu_{7})(z-\kappa_{1}/\nu_{8}) y(qz)+\frac{(q\nu_{1}-z)(q\nu_{2}-z)(q\nu_{3}-z)(q\nu_{4}-z)}{q^2} y(z/q) \nonumber \\
&+\Bigl[(1+1/q)z^{3} - \Bigl( \nu_{1}+ \nu_{2}+ \nu_{3}+ \nu_{4}+\frac{q\kappa_{1}}{\nu_{7}}+\frac{\kappa_{1}}{\nu_{8}}\Bigr)z^{2}\nonumber \\
&\quad + q \Bigl\{ \frac{(1- \nu_{1}\nu_{7}/\kappa_{1})(1- \nu_{2}\nu_{7}/\kappa_{1})(1- \nu_{3}\nu_{7}/\kappa_{1})(1- \nu_{4}\nu_{7}/\kappa_{1})} {(g - \nu_{7}/\kappa_{1})\nu_{7}/\kappa_{1}}  +\frac{\kappa _1^2}{\nu _7 \nu _8} \nonumber \\
& \qquad   + \frac{\nu_{1}\nu_{2}\nu_{3}\nu_4 \nu_7 }{\kappa _1} \Bigl( - \frac{1}{\nu _1} - \frac{1}{\nu _2} - \frac{1}{\nu _3} - \frac{1}{\nu _4} + \frac{\nu _5}{\kappa _2}+ \frac{\nu _6}{\kappa _2}+ \frac{\nu _7}{\kappa _1} \Bigr) + \sum _{1\leq i <j \leq 4}\nu _i \nu _j  \Bigr\} z \nonumber \\
&\quad -\frac{q \kappa_{1}^2 \kappa_{2} (\nu_5+\nu_6)}{\nu_5 \nu_6 \nu_7 \nu_8 } \Bigr\}\Bigr] y (z)=0 \, . \nonumber
\end{align}
By applying a gauge transformation, we obtain the variant of the $q$-Heun equation of degree $3$, and we may regard $g$ as an accessory parameter.
We have a similar result by substituting $f= \kappa_{1}/\nu_{8} $ to the equation $L_1 y(z)=0 $.

We substitute $f= \nu_{1} $ to the equation $L_1 y(z)=0 $.
Then we have
\begin{align}
& q (z- \kappa_{1}/\nu_{7})(z-\kappa_{1}/\nu_{8}) y(qz)+\frac{(z- \nu_{1})(z-q\nu_{2})(z-q\nu_{3})(z- q\nu_{4})}{q}y(z/q) \nonumber \\
&+\Bigl[(q +1 )z^{3} - \Bigl(\nu_{1}+ q\nu_{2}+ q\nu_{3}+ q\nu_{4}+\frac{q\kappa_{1}}{\nu_{7}}+\frac{q\kappa_{1}}{\nu_{8}}\Bigr)z^{2}\nonumber \\
&\quad + q \Bigl\{ \frac{ - (\nu_1 -\kappa_{1} /\nu _7) ( \nu_1-\kappa_{1} /\nu _8 )}{ g \nu_1 } +\nu _1 ( \nu _2 + \nu _3 + \nu _4 ) \nonumber \\
& \qquad \qquad + \frac{\kappa_{1}^2}{\nu _7 \nu _8} + \frac{\kappa_{1}^2}{\nu _1 \nu _7 \nu _8} \Bigl( \frac{\kappa _2}{\nu _5}+ \frac{\kappa _2}{\nu _6} \Bigr) \Bigr\} z -\frac{q^2 \nu_1 \nu_2 \nu_3 \nu_4 (\nu_5+\nu_6)}{\kappa_{2}} \Bigr\}\Bigr] y (z)=0 \, . \nonumber
\end{align}
By applying a gauge transformation, we obtain the variant of the $q$-Heun equation of degree $3$, and we may regard $g$ as an accessory parameter.
We have similar results by substituting $f= \nu_{i} $ $(i = 2,3,4)$  to the equation $L_1 y(z)=0 $.

\subsection{The case $E^{(1)}_{7}$} $ $ 

The $q$-Painlev\'e equation of type $E^{(1)}_{7} $ ($q$-$P(E^{(1)}_{7})$) in \cite{KNY} was given as
\begin{align}
& \frac{(fg- \kappa_{1}/\kappa_{2} ) (\overline{f}g- \kappa_{1}/(q\kappa_{2}) )}{(fg-1)(\overline{f}g-1)}=\frac{(g- \nu_{5}/\kappa_{2} ) (g- \nu_{6}/\kappa_{2} ) (g- \nu_{7}/\kappa_{2} ) (g- \nu_{8}/\kappa_{2} )}{ (g - 1/\nu_{1} )(g - 1/\nu_{2} )(g - 1/\nu_{3} )(g - 1/\nu_{4} )}, \nonumber \\
& \frac{(fg-\kappa_{1}/\kappa_{2} ) (f\underline{g}- q\kappa_{1}/\kappa_{2} )}{(fg-1)(f\underline{g}-1)}=\frac{(f- \kappa_{1}/\nu_{5})(f- \kappa_{1}/\nu_{6})(f- \kappa_{1}/\nu_{7})(f- \kappa_{1}/\nu_{8})}{(f-\nu_{1})(f-\nu_{2})(f-\nu_{3})(f-\nu_{4})}  . \nonumber 
\end{align}
The space of initial condition for $q$-$P(E^{(1)}_{7})$ was realized by blowing up eight points $P_{1}, \dots , P_8 $ of ${\mathbb P^1}\times {\mathbb P^1}$, where 
\begin{equation}
    P_{i}:\Bigl(\nu_{i},\frac{1}{\nu_{i}}\Bigr)_{i=1,2,3,4},\quad \Bigl(\frac{\kappa_{1}}{\nu_{i}},\frac{\nu_{i}}{\kappa_{2}}\Bigr)_{i=5,6,7,8} \, .
\nonumber 
\end{equation}
On the other hand, $q$-$P(E^{(1)}_{7})$ is obtained by the compatibility condition for the Lax operators $L_1$ and $L_2$, where
\begin{align}
& L_{1}=\frac{q(\kappa_{1}-\kappa_{2})(g\kappa_{2}-\nu_{5})(g\kappa_{2}-\nu_{6})(g\kappa_{2}-\nu_{7})(g\kappa_{2}-\nu_{8})}{g\kappa_{1}{\kappa_{2}}^{2}(fg\kappa_{2}-\kappa_{1})(g\kappa_{2}z-\kappa_{1})} \nonumber \\
& \quad -\frac{q(\kappa_{1}-\kappa_{2})(g\nu_{1}-1)(g\nu_{2}-1)(g\nu_{3}-1)(g\nu_{4}-1)}{g(fg-1)\kappa_{1}\nu_{1}\nu_{2}\nu_{3}\nu_{4}(gz-q)} \nonumber \\
& \quad +\frac{(q\nu_{1}-z)(q\nu_{2}-z)(q\nu_{3}-z)(q\nu_{4}-z)\{(g\kappa_{2}z-\kappa_{1}q)- \kappa_{1} (gz-q)T_{z}^{-1}\}}{q \kappa_{1}\nu_{1}\nu_{2}\nu_{3}\nu_{4}(fq-z)z^{2}(q-gz)} \nonumber \\
& \quad -\frac{q(\kappa_{1}-\nu_{5}z)(\kappa_{1}-\nu_{6}z)(\kappa_{1}-\nu_{7}z)(\kappa_{1}-\nu_{8}z)\{\kappa_{1}(gz-1)-(g\kappa_{2}z-\kappa_{1})T_{z}\}}{{\kappa_{1}}^{4}(f-z)z^{2}(g\kappa_{2}z -\kappa_{1})} \, , \nonumber \\
& L_{2}=(1-zg \kappa_{2}/\kappa_{1})T_{z}-(1-zg)+z(z-f)gT \, .\nonumber 
\end{align}
We track the linear differential equation $L_1 y(z)= 0$ on the process of the blow-up of the point $P_i$ $(i=1,2,\dots ,8)$ respectively, and we investigate a relationship with the variant of the $q$-Heun equation of degree $4$;
\begin{align}
&(x-h_{1}q^{1/2})(x-h_{2}q^{1/2})(x-h_{3}q^{1/2})(x-h_{4}q^{1/2})g(x/q) \label{eq:qHeunE17} \\
&+(x-l_{1}q^{-1/2})(x-l_{2}q^{-1/2})(x-l_{3}q^{-1/2})(x-l_{4}q^{-1/2})g(qx)\nonumber \\
&+\Bigl[-(q^{1/2}+q^{-1/2})x^{4}+(h_{1}+h_{2} +h_{3}+h_{4}+l_{1}+l_{2}+l_{3}+l_{4})x^{3}+Ex^{2}\nonumber \\
&+(h_{1}h_{2}h_{3}h_{4}l_{1}l_{2}l_{3}l_{4})^{1/2}\bigl\{(h_{1}^{-1}+h_{2}^{-1}+h_{3}^{-1}+h_{4}^{-1}+l_{1}^{-1} +l_{2}^{-1}+l_{3}^{-1}+l_{4}^{-1})x \nonumber \\
&-(q^{1/2}+q^{-1/2})\bigr\}\Bigr]g(x)=0 \, . \nonumber 
\end{align}
Note that detailed calculation in the following subsections was performed in \cite{Ssk}.

\subsubsection{The points $P_{i}:(\nu_{i},1/\nu_{i})$ $(i =1,2,3,4 )$} $ $

We realize the blow-up at $(f,g)= (\nu_{1},1/\nu_{1})$ by setting $(f,g)=(f_{1}g_{1}+\nu_{1},g_{1}+1/\nu_{1})$.
Then the point $P_{1}:(\nu_{1},1/\nu_{1})$ corresponds to the line $g_{1}=0$.
We substitute $(f,g)=(f_{1}g_{1}+\nu_{1},g_{1}+1/\nu_{1})$ into the equation $L_1 y(z)=0 $, and set $g_{1}=0$.
Then we have
\begin{align}
 &\frac{q(\kappa_{1}-\nu_{5}z)(\kappa_{1}-\nu_{6}z)(\kappa_{1}-\nu_{7}z)(\kappa_{1}-\nu_{8}z)}{{\kappa_{1}}^{4}(\nu_{1}-z)z^{2}}y(qz) \nonumber \\
& +\frac{(q\nu_{2}-z)(q\nu_{3}-z)(q\nu_{4}-z)}{q\nu_{1}\nu_{2}\nu_{3}\nu_{4}z^{2}} y(z/q) \nonumber \\
&+\Bigl\{ -\frac{q (\kappa_{1}-\kappa_{2}) (\nu_{1}-\nu_{2})(\nu_{1}-\nu_{3})(\nu_{1}-\nu_{4})}{\kappa_{1}(f_{1}+{\nu_{1}}^{2})\nu_{2}\nu_{3}\nu_{4}(q\nu_{1}-z)}\nonumber \\
& \quad +\frac{q(\kappa_{2}-\nu_{1}\nu_{5})(\kappa_{2}-\nu_{1}\nu_{6})(\kappa_{2}-\nu_{1}\nu_{7})(\kappa_{2}-\nu_{1}\nu_{8})}{\kappa_{1}{\kappa_{2}}^{2}{\nu_{1}}^{2}(\kappa_{1}\nu_{1}-\kappa_{2}z)} \nonumber \\
& \quad -\frac{(q\nu_{2}-z)(q\nu_{3}-z)(q\nu_{4}-z)(\kappa_{1}q\nu_{1}-\kappa_{2}z)}{\kappa_{1}q\nu_{1}\nu_{2}\nu_{3}\nu_{4}(q\nu_{1}-z)z^{2}} \nonumber \\
& \quad -\frac{q(\kappa_{1}-\nu_{5}z)(\kappa_{1}-\nu_{6}z)(\kappa_{1}-\nu_{7}z)(\kappa_{1}-\nu_{8}z)}{{\kappa_{1}}^{3}z^{2}(\kappa_{1}\nu_{1} -\kappa_{2}z)} \Bigr\} y(z) =0  \, .\nonumber 
\end{align}
Let $\lambda $ be the parameter such that $q^{\lambda } = q^{3/2} \kappa _2/ \kappa _1 $.
Then the function $u(z)= z^{\lambda} y(z)/(q\nu_{1}-z)$ satisfies 
\begin{align}
&\Bigl(z-\frac{\kappa_{1}}{\nu_{5}}\Bigr)\Bigl(z-\frac{\kappa_{1}}{\nu_{6}}\Bigr)\Bigl(z-\frac{\kappa_{1}}{\nu_{7}}\Bigr)\Bigl(z-\frac{\kappa_{1}}{\nu_{8}}\Bigr)u(qz) \nonumber \\
&+(z-q^{2}\nu_{1})(z-q\nu_{2})(z-q\nu_{3})(z-q\nu_{4})u(z/q) \nonumber \\
&+\Bigl[-(q^{1/2}+q^{-1/2})z^{4} + q^{1/2} \Bigl(q \nu_{1}+\nu_{2}+ \nu_{3} + \nu_{4}+\frac{\kappa_{1}}{\nu_{5}}+\frac{\kappa_{1}}{\nu_{6}}+\frac{\kappa_{1}}{\nu_{7}} +\frac{\kappa_{1}}{\nu_{8}}\Bigr)z^{3} \nonumber \\
& \quad  -q^{3/2} \Bigl( \frac{K_1}{\kappa_2 (f_1+\nu_1^2)} +K_2 \Bigr) z^{2}+\frac{q^{3}\kappa_{1}\nu_{1}\nu_{2}\nu_{3}\nu_{4}}{\kappa_{2}} \Bigl\{ q^{-1/2} \Bigl( \frac{1}{q \nu_{1}}+\frac{1}{\nu_{2}}+\frac{1}{\nu_{3}}+\frac{1}{\nu_{4}} \nonumber \\
& \quad +\frac{\nu_{5}}{\kappa_{1}}+\frac{\nu_{6}}{\kappa_{1}}+\frac{\nu_{7}}{\kappa_{1}}+\frac{\nu_{8}}{\kappa_{1}}\Bigr)z-(q^{1/2}+q^{-1/2})\Bigr\}\Bigr]u(z)=0 \, , \nonumber 
\end{align}
where
\begin{align}
& K_1=\nu_1 (\kappa_1-\kappa_2) (\nu_1 -\nu_2) (\nu_1 -\nu_3) (\nu_1 -\nu_4) ,\nonumber \\
& K_2= \kappa_1 \nu_1 \Bigl( \frac{-\nu_1+\nu_2+\nu_3+\nu_4}{\kappa_2}+\frac{1}{\nu_5}+\frac{1}{\nu_6}+\frac{1}{\nu_7}+\frac{1}{\nu_8} \Bigr) \nonumber \\
& \qquad + \nu_2 \nu_3 \nu_4 \Bigl( -\frac{1}{\nu_1}+\frac{1}{\nu_2} +\frac{1}{\nu_3} +\frac{1}{\nu_4} +\frac{\nu_5+\nu_6+\nu_7+\nu_8}{\kappa_2}\Bigr) . \nonumber
\end{align}

Note that we used the relation $\kappa_{1}^2 \kappa_{2}^2 =q \nu_{1} \nu_{2} \dots \nu_{8} $.

We compare it with the variant of the $q$-Heun equation of degree $4$ given in Eq.~(\ref{eq:qHeunE17}).
Then we obtain the following correspondence
\begin{align}
&h_{1}=q^{3/2}\nu_{1},\quad h_{2}=q^{1/2}\nu_{2},\quad h_{3}=q^{1/2}\nu_{3},\quad h_{4}=q^{1/2}\nu_{4}, \nonumber \\
&l_{1}=\frac{q^{1/2}\kappa_{1}}{\nu_{5}},\quad l_{2}=\frac{q^{1/2}\kappa_{1}}{\nu_{6}},\quad l_{3}=\frac{q^{1/2}\kappa_{1}}{\nu_{7}},\quad l_{4}=\frac{q^{1/2}\kappa_{1}}{\nu_{8}} . \nonumber 
\end{align}
We may regard $f_1$ as an accessory parameter.

On the blow-up at the point $P_{i}$ $(i=2,\, 3,\, 4)$, we obtain the equation where the parameters $\nu _1$ and $\nu _i$ are exchanged.

\subsubsection{The points $P_{i}:(\kappa_{1}/\nu_{i}, \nu_{i}/\kappa_{2})$ $(i =5,6,7,8 )$} $ $

We realize the blow-up at $(f,g)= (\kappa_{1}/\nu_{5}, \nu_{5}/\kappa_{2})$ by setting $(f,g)=(f_{1}g_{1}+\kappa_{1}/\nu_{5},g_{1}+\nu_{5}/\kappa_{2})$.
Then the point $P_{5}:(\kappa_{1}/\nu_{5}, \nu_{5}/\kappa_{2})$ corresponds to the line $g_{1}=0$.
We substitute $(f,g)= (f_{1}g_{1}+\kappa_{1}/\nu_{5},g_{1}+\nu_{5}/\kappa_{2})$ into the equation $L_1 y(z)=0 $, and set $g_{1}=0$.
Then we have
\begin{align}
&\frac{q\nu_{5}(\kappa_{1}-\nu_{6}z)(\kappa_{1}-\nu_{7}z)(\kappa_{1}-\nu_{8}z)}{{\kappa_{1}}^{4}z^{2}} y(qz) \nonumber  \\
&+ \frac{\nu_{5}(q\nu_{1}-z)(q\nu_{2}-z)(q\nu_{3}-z)(q\nu_{4}-z)}{q\nu_{1}\nu_{2}\nu_{3}\nu_{4}z^{2}(\kappa_{1}q -\nu_{5}z)} y(z/q)  \nonumber \\
&+\Bigl\{-\frac{q (\kappa_{1}-\kappa_{2})(\nu_{5}-\nu_{6})(\nu_{5}-\nu_{7})(\nu_{5}-\nu_{8})}{\kappa_{1}(\kappa_{1}\kappa_{2}+f_{1}{\nu_{5}}^{2})(\kappa_{1}-\nu_{5}z)} \nonumber \\
& \quad - \frac{\kappa_{2}{\nu_{5}}(q\nu_{1}-z)(q\nu_{2}-z)(q\nu_{3}-z)(q\nu_{4}-z)}{q \kappa_{1} \nu_{1}\nu_{2}\nu_{3}\nu_{4}z^2 (\kappa_{2}q-\nu_{5}z)} \nonumber \\
& \quad +\frac{q(\kappa_{2}-\nu_{1}\nu_{5})(\kappa_{2}-\nu_{2}\nu_{5})(\kappa_{2}-\nu_{3}\nu_{5})(\kappa_{2}-\nu_{4}\nu_{5})}{\kappa_{1}\kappa_{2}\nu_{1}\nu_{2}\nu_{3}\nu_{4}\nu_{5}(\kappa_{2}q-\nu_{5}z)} \nonumber \\
& \quad -\frac{q\nu_{5}(\kappa_{2}-\nu_{5}z) (\kappa_{1}-\nu_{6}z)(\kappa_{1}-\nu_{7}z)(\kappa_{1}-\nu_{8}z)}{{\kappa_{1}}^{3}\kappa _2 z^{2}(\kappa_{1}-\nu_{5}z)} \Bigr\} y(z)=0 \, . \nonumber 
\end{align}
Let $\lambda $ be the parameter such that $q^{\lambda } = q^{3/2} \kappa _2/ \kappa _1 $.
Then the function $u(z)= z^{\lambda} y(z)/(\kappa_{1}-\nu_{5}z)$ satisfies 
\begin{align}
&(z-q\nu_{1})(z-q\nu_{2})(z-q\nu_{3})(z-q\nu_{4})u(z/q) \nonumber \\
&+\Bigl(z-\frac{\kappa_{1}}{q\nu_{5}}\Bigr)\Bigl(z-\frac{\kappa_{1}}{\nu_{6}}\Bigr)\Bigl(z-\frac{\kappa_{1}}{\nu_{7}}\Bigr)\Bigl(z-\frac{\kappa_{1}}{\nu_{8}}\Bigr)u(qz) \nonumber \\
&+\Bigl[-(q^{1/2}+q^{-1/2})z^{4}+ q^{1/2} \Bigl( \nu_{1}+\nu_{2}+\nu_{3}+\nu_{4}+\frac{\kappa_{1}}{q \nu_{5}}+\frac{\kappa_{1}}{\nu_{6}}+\frac{\kappa_{1}}{\nu_{7}} +\frac{\kappa_{1}}{\nu_{8}}\Bigr)z^{3} \nonumber \\
&- q^{1/2} \Bigl( \frac{K_1}{\kappa_{2} \nu_{5} (f_{1}{\nu_{5}}^{2} + \kappa_{1} \kappa_{2})}+K_2 \Bigr) z^{2}+\frac{q^{2}\kappa_{1}\nu_{1}\nu_{2}\nu_{3}\nu_{4}}{\kappa_{2}}\Bigl\{ q^{-1/2} \Bigl( \frac{1}{\nu_{1}}+\frac{1}{\nu_{2}} +\frac{1}{\nu_{3}} +\frac{1}{\nu_{4}} \nonumber \\
& +\frac{q\nu_{5}}{\kappa_{1}}+\frac{\nu_{6}}{\kappa_{1}}+\frac{\nu_{7}}{\kappa_{1}}+\frac{\nu_{8}}{\kappa_{1}}\Bigr)z-(q^{1/2}+q^{-1/2})\Bigr\}\Bigr]u(z)=0 \, , \nonumber 
\end{align}
where
\begin{align}
& K_1=q \nu_1 \nu_2 \nu_3 \nu_4 (\kappa_1-\kappa_2) (\nu_5 -\nu_6) (\nu_5 -\nu_7) (\nu_5 -\nu_8) , \nonumber \\
& K_2= \frac{\kappa _1 \kappa _2}{\nu _5 } \Bigl( \frac{\nu_1+\nu_2+\nu_3+\nu_4}{\kappa_2}-\frac{1}{\nu_5}+\frac{1}{\nu_6}+\frac{1}{\nu_7}+\frac{1}{\nu_8} \Bigr) \nonumber \\
& \qquad + q \nu_1 \nu_2 \nu_3 \nu_4 \frac{\nu _5 }{\kappa_2} \Bigl( \frac{1}{\nu_1}+\frac{1}{\nu_2} +\frac{1}{\nu_3} +\frac{1}{\nu_4} +\frac{-\nu_5+\nu_6+\nu_7+\nu_8}{\kappa_2}\Bigr) . \nonumber
\end{align}
Note that we used the relation $\kappa_{1}^2 \kappa_{2}^2 =q \nu_{1} \nu_{2} \dots \nu_{8} $.

We compare it with the variant of the $q$-Heun equation of degree $4$ given in Eq.~(\ref{eq:qHeunE17}).
Then we obtain the following correspondence
\begin{align}
    &h_{1}=q^{1/2}\nu_{1},\quad h_{2}=q^{1/2}\nu_{2},\quad h_{3}=q^{1/2}\nu_{3},\quad h_{4}=q^{1/2}\nu_{4}, \nonumber \\
    &l_{1}=\frac{\kappa_{1}}{q^{1/2}\nu_{5}},\quad l_{2}=\frac{q^{1/2}\kappa_{1}}{\nu_{6}},\quad l_{3}=\frac{q^{1/2}\kappa_{1}}{\nu_{7}},\quad l_{4}=\frac{q^{1/2}\kappa_{1}}{\nu_{8}} . \nonumber 
\end{align}
We may regard $f_1$ as an accessory parameter.

On the blow-up at the point $P_{i}$ $(i=6,\, 7,\, 8)$, we obtain the equation where the parameters $\nu _5$ and $\nu _i$ are exchanged.

\subsubsection{The other cases related to the variant of the $q$-Heun equation of degree four} $ $

We substitute $f= \nu_{1} $ to the equation $L_1 y(z)=0 $.
Let $\lambda $ be the parameter such that $q^{\lambda } = q^{1/2} \kappa _2/ \kappa _1 $.
Then the function $u(z)= z^{\lambda} y(z)$ satisfies
\begin{align}
&\Bigl(z-\frac{\kappa_{1}}{\nu_{5}}\Bigr)\Bigl(z-\frac{\kappa_{1}}{\nu_{6}}\Bigr)\Bigl(z-\frac{\kappa_{1}}{\nu_{7}}\Bigr)\Bigl(z-\frac{\kappa_{1}}{\nu_{8}}\Bigr)u(qz) \nonumber \\
& +(z-\nu_{1})(z-q\nu_{2})(z-q\nu_{3})(z-q\nu_{4})u(z/q) \nonumber \\
&+\Bigl[-(q^{1/2}+q^{-1/2})z^{4} + q^{1/2} \Bigl(\frac{\nu_{1}}{q} +\nu_{2}+\nu_{3}+\nu_{4}+\frac{\kappa_{1}}{\nu_{5}}+\frac{\kappa_{1}}{\nu_{6}}+\frac{\kappa_{1}}{\nu_{7}} +\frac{\kappa_{1}}{\nu_{8}}\Bigr)z^{3}\nonumber \\
& \quad  -q^{1/2} \Bigl( \frac{(\nu _1- \kappa _1/\nu _5)(\nu _1- \kappa _1/\nu _6)(\nu _1- \kappa _1/\nu _7)(\nu _1- \kappa _1/\nu _8)(\kappa_1-\kappa_2)}{\nu_1^2 ( g \nu_1 \kappa _2 -\kappa _1)} +K_2 \Bigr) z^{2} \nonumber \\
& \quad + \frac{q^{2}\kappa_{1}\nu_{1}\nu_{2}\nu_{3}\nu_{4}}{\kappa_{2}}\Bigl\{ q^{-1/2} \Bigl(\frac{q}{\nu_{1}}+\frac{1}{\nu_{2}}+\frac{1}{\nu_{3}} +\frac{1}{\nu_{4}}+\frac{\nu_{5}}{\kappa_{1}}+\frac{\nu_{6}}{\kappa_{1}}+\frac{\nu_{7}}{\kappa_{1}}+\frac{\nu_{8}}{\kappa_{1}}\Bigr) z \nonumber \\
& \qquad \qquad \qquad -(q^{1/2}+q^{-1/2})\Bigr\}\Bigr]u(z)=0 \, , \nonumber 
\end{align}
where $K_2$ is a value which does not depend on $z$ nor $g$.
Hence we obtain the variant of the $q$-Heun equation of degree $4$, and we may regard $g$ as an accessory parameter.
We have similar results by substituting $f= \nu_{i} $ $(i = 2,3,4)$  to the equation $L_1 y(z)=0 $.

We substitute $f= \kappa_{1}/\nu_{5} $ to the equation $L_1 y(z)=0 $.
Let $\lambda $ be the parameter such that $q^{\lambda } = q^{1/2} \kappa _2/ \kappa _1 $.
Then we have
\begin{align}
& (z-q\nu_{1})(z-q\nu_{2})(z-q\nu_{3})(z-q\nu_{4})u(z/q) \nonumber \\
& +\Bigl(z-\frac{q\kappa_{1}}{\nu_{5}}\Bigr)\Bigl(z-\frac{\kappa_{1}}{\nu_{6}}\Bigr)\Bigl(z-\frac{\kappa_{1}}{\nu_{7}}\Bigr)\Bigl(z-\frac{\kappa_{1}}{\nu_{8}}\Bigr)u(qz) \nonumber \\
& +\Bigl[-(q^{1/2}+q^{-1/2})z^{4}+ q^{1/2} \Bigl(\nu_{1}+\nu_{2}+\nu_{3}+\nu_{4}+\frac{q\kappa_{1}}{\nu_{5}}+\frac{\kappa_{1}}{\nu_{6}}+\frac{\kappa_{1}}{\nu_{7}} +\frac{\kappa_{1}}{\nu_{8}}\Bigr)z^{3} \nonumber \\
& - q^{3/2} \Bigl( \frac{(\kappa_1-\kappa_2)(\kappa_1-\nu _1 \nu_5 )(\kappa_1-\nu _1 \nu_6 )(\kappa_1-\nu _1 \nu_7 )(\kappa_1-\nu _1 \nu_8 )}{\kappa_{1}^2 \kappa_{2} \nu_{5} (\nu_{5} - \kappa_{1} g)}+K_2 \Bigr) z^{2} \nonumber \\
& +\frac{q^{3}\kappa_{1}\nu_{1}\nu_{2}\nu_{3}\nu_{4}}{\kappa_{2}}\Bigl\{ q^{-1/2} \Bigl(\frac{1}{\nu_{1}}+\frac{1}{\nu_{2}} +\frac{1}{\nu_{3}} +\frac{1}{\nu_{4}}+\frac{\nu_{5}}{q\kappa_{1}}+\frac{\nu_{6}}{\kappa_{1}}+\frac{\nu_{7}}{\kappa_{1}}+\frac{\nu_{8}}{\kappa_{1}}\Bigr)z\nonumber \\
& \qquad \qquad \qquad -(q^{1/2}+q^{-1/2})\Bigr\}\Bigr]u(z)=0 \, . \nonumber 
\end{align}
where $K_2$ is a value which does not depend on $z$ nor $g$.
Hence we obtain the variant of the $q$-Heun equation of degree $4$, and we may regard $g$ as an accessory parameter.
We have similar results by substituting $f= \kappa_{1}/\nu_{i} $ $(i = 6,7,8)$ to the equation $L_1 y(z)=0 $.

\section{On associated linear $q$-difference equation related with $q$-$P(E^{(1)}_8 )$} \label{sec:qPE8}

We review a Lax pair of the $q$-Painlev\'e equation of type $E^{(1)}_8 $ ($q$-$P(E^{(1)}_8)$) by following Yamada \cite{Y}.
Let $u_1, u_2, \dots, u_8$, $h_1$ and $h_2$ be the non-zero parameters which satisfy the condition $h_1 ^2 h_2 ^2 = q u_1 u_2 \dots u_8$.
Set 
\begin{equation}
U(z)= \prod _{i=1}^8 (z-u_i) ,
\nonumber 
\end{equation}
and define the quartic polynomials $P_n(x)$ and $P_d (x)$ by the condition 
\begin{align}
& \Bigl(z-\frac{h_2}{z} \Bigr) P_n \Bigl(z +\frac{h_2}{z} \Bigr) =  \frac{U(z)}{z^3} -\Bigl(\frac{z}{h_2}  \Bigr)^3 U\Bigl( \frac{h_2}{z}  \Bigr), \label{eq:PnPd} \\
& \Bigl(z-\frac{h_2}{z}  \Bigr) P_d \Bigl(z +\frac{h_2}{z} \Bigr) = z^5 U\Bigl( \frac{h_2}{z}  \Bigr) - \Bigl( \frac{h_2}{z}  \Bigr)^5 U(z) . \nonumber
\end{align}
Set
\begin{align}
& \psi _n (f_1, f_2 ;g ) = (f_1-g) (f_2-g) -\Bigl(\frac{h_1}{q} -h_2  \Bigr) (h_1-h_2 )\frac{1}{h_2}, \nonumber \\
& \psi _d (f_1, f_2 ;g ) = \Bigl(\frac{f_1}{h_1}q-\frac{g}{h_2} \Bigr) \Bigl(\frac{f_2}{h_1}-\frac{g}{h_2} \Bigr)- \Bigl(\frac{q}{h_1} -\frac{1}{h_2}  \Bigr) \Bigl(\frac{1}{h_1} -\frac{1}{h_2}  \Bigr) h_2 , \nonumber \\
& V(f_1, f_2;g) = q \psi _n (f_1, f_2 ;g ) P_d (g) - h_1^2 h_2^4 \psi _d (f_1, f_2 ;g ) P_n (g), \nonumber \\
& f(u)=u+\frac{h_1}{u}, \; g(u)=u + \frac{h_2}{u}, \; \overline{f}(u)= u+\frac{h_1}{qu}, \; \overline{g}(u)=u + \frac{q h_2}{u}, \nonumber \\
& \varphi ( f, g ) = (f -g ) \Bigl(\frac{f}{h_1} -\frac{g}{h_2}  \Bigr) -( h_1 -h_2)\Bigl(\frac{1}{h_1} - \frac{1}{h_2}  \Bigr).\nonumber 
\end{align}
A Lax pair for the $q$-Painlev\'e equation of type $E^{(1)}_8$ introduced in \cite{Y} is written as
\begin{align}
& L_1 = \frac{q^5 U(z/q)}{(z^2 - h_1 q^2) (f-f(z/q) )}\Bigl[ T _{z}^{-1}  -\frac{g-g(q h_1 /z)}{g- g(z/q)} \Bigr] \nonumber \\
& \quad + \frac{z^8 U(h_1 /z)}{h_1 ^4 (z^2-h_1) (f-f(z))}\Bigl[ T_z-\frac{g-g(z)}{g- g(h_1 /z)} \Bigr] \nonumber \\
& \quad + \frac{(h_1-h_2)z^2 (z^2 -q h_1 ) V( \overline{f}(z/q ) ,f ;g)}{q h_1^3 h_2^3 g \varphi (f,g) (g-g (h_1/z)) (g-g (z/q)) } , \nonumber \\
& L_2  = ( g-g(z/q)) T_z^{-1} - ( g-g(q h_1/ z) ) + ( f-f(z/q) ) (h_1/z - z/q^2 ) T \cdot T_z^{-1} . \nonumber
\end{align}
Here $T_{z}$ represents the transformation $z\rightarrow qz$ and $T$ represents the time evolution.
The $q$-Painlev\'e equation of type $E^{(1)}_8 $ ($q$-$P(E^{(1)}_8)$) for the variables $f$ and $g$ was obtained by the compatibility condition for the operators $L_1$ and $L_2$ in \cite{Y}, and the expression of $q$-$P(E^{(1)}_8)$ was also described there.
It seems that the definitions of the equation $q$-$P(E^{(1)}_8)$ and the Lax pair in \cite{KNY} is less explicit than those in \cite{KNY}, and we use the definition in \cite{Y}.

We pick up the $q$-difference equation $L_1 y(z)=0$,~i.~e.
\begin{align}
& \frac{q^5 U(z/q)}{(z^2 - h_1 q^2) (f-f(z/q) )}\Bigl[ y(z/q)  -\frac{g-g(q h_1 /z)}{g- g(z/q)} y(z) \Bigr] \label{eq:L1E8}\\
& \quad + \frac{z^8 U(h_1 /z)}{h_1 ^4 (z^2-h_1) (f-f(z))}\Bigl[ y(qz) -\frac{g-g(z)}{g- g(h_1 /z)} y(z) \Bigr] \nonumber \\
& \quad + \frac{(h_1-h_2)z^2 (z^2 -q h_1 ) V( \overline{f}(z/q ) ,f ;g)}{q h_1^3 h_2^3 g \varphi (f,g)  \{ g-g (h_1/z) \} \{ g-g (z/q) \} } y(z) = 0 , \nonumber 
\end{align}
and consider a limit with respect to the parameters $f$ and $g$.
Namely, set 
\begin{equation}
f= f(u_1 +\varepsilon c_1 ) = u_1 +\varepsilon c_1 + \frac{h_1}{ u_1 +\varepsilon c_1}, \; g= g ( u_1 +\varepsilon c_2 ) = u_1 +\varepsilon c_2 + \frac{h_2}{ u_1 +\varepsilon c_2} \label{eq:fg}
\end{equation}
and consider the limit $\varepsilon \to 0$ in Eq.~(\ref{eq:L1E8}).
Then
\begin{equation}
f -f (z)= -\frac{(z - u_1) (u_1 z-h_1)}{u_ 1 z} +O(\varepsilon  ) , \; g -g (z)= -\frac{(z - u_1) (u_1 z-h_2)}{u_ 1 z} +O(\varepsilon  )  \nonumber 
\end{equation}
as $\varepsilon \to 0 $.
Write 
\begin{equation}
 U_7(z)= \prod _{i=2}^8 (z -u_i).  \nonumber 
\end{equation}
Then $U(z)= (z-u_1) U_7(z)$.
On the coefficients of $y(z/q)$ and $ y(qz)$ in Eq.~(\ref{eq:L1E8}), we have
\begin{align}
& \frac{q^5 U(z/q)}{(z^2 - h_1 q^2) (f-f(z/q) )} = \frac{ - q^5 u_1 z U_7(z/q)}{(z^2 - h_1 q^2) (u_1 z-q h_1 ) } +O(\varepsilon  ) ,  \nonumber \\
& \frac{z^8 U(h_1 /z)}{h_1 ^4 (z^2-h_1) (f-f(z))} = \frac{z^8 u_1 U_7 (h_1 /z)}{h_1 ^4 (z^2-h_1) (z - u_1) }+O(\varepsilon  ) . \nonumber 
\end{align}
The limit of the coefficient of $ y(z)$ is more complicated.
By applying a gauge transformation for $y(z)$ and calculating the limit of Eq.~(\ref{eq:L1E8}) as $\varepsilon \to 0 $, we obtain the following theorem.
\begin{theorem}
Set
\begin{equation}
y(z)= \frac{(z-q u_1) (u_1 z - h_1)}{z} \tilde{y} (z).
\label{eq:yztildeyz}
\end{equation}
Write $f$ and $g$ as Eq.~(\ref{eq:fg}).
Then Eq.~(\ref{eq:L1E8}) tends to
\begin{align}
& \quad B^{-} (z)  \tilde{y} (z/q) +  B^{+} (z) \tilde{y} (qz )+ B^{0} (z) \tilde{y} (z) = \Bigl( C_0 \frac{c_2}{c_1 -c_2 } + C_0 ' \Bigr) \tilde{y} (z) \label{eq:B-B+B0eq} 
\end{align}
as $\varepsilon \to 0$, where 
\begin{align}
& B^{-} (z) = \frac{ (z -q^2 u_1) }{q ^2 h_1 z^2 (z^2 -q h_1 )(z^2 - q^2 h_1 ) } \prod _{j=2}^8 (z -q u_j) , \nonumber \\
& B^{+} (z) =  \frac{  q (q u_1 z  - h_1) }{ h_1 ^5 z^2 (z^2 -q h_1 )(z^2-h_1) } \prod _{j=2}^8 ( u_j z  - h_1) \nonumber \\
& B^{0} (z) = \frac{-q z (h_1^{1/2} - qu_1) }{2 h_1^{7/2} (z - h_1^{1/2})(z - qh_1^{1/2})}  \prod _{j=2}^8 (h_1^{1/2} - u_j) \nonumber   \\
& \qquad + \frac{q z (h_1^{1/2} + qu_1) }{2 h_1^{1/2} (z + h_1^{1/2})(z + qh_1^{1/2})} \prod _{j=2}^8 (h_1^{1/2} + u_j) - \frac{q^{3/2}}{h_1} ( u_1 u_2 u_3 u_4 u_5 u_6 u_7 u_8 )^{1/2} \nonumber \\
& \qquad \cdot \Bigl[ (q+1) \Bigl(\frac{z^2}{(q h_1)^2} +\frac{1}{z^{2}}\Bigr) - \Bigl(\frac{z}{q h_1} + \frac{1}{z} \Bigr) \Bigl\{ q \frac{u_1}{h_1} + \frac{1}{q u_1} + \sum _{j=2}^8 \Bigl(\frac{u_j}{h_1} + \frac{1}{u_j} \Bigr) \Bigr\} \Bigr] , \nonumber \\
& C_0 = \frac{q (h_2-h_1) }{ h_1^2 u_1 (h_1 -u_1^2) (h_2 -u_1^2) }\prod _{j=2}^8  (u_1 -u_j )  \nonumber 
\end{align}
and $  C_0 '$ is the constant which does not depend on $z$, $c_1$ and $c_2$.
\end{theorem}
\begin{proof}
As $\varepsilon \to 0$, we have
\begin{align}
& \varphi (f,g) = -\frac{(h_1-h_2)  (h_1-u_1^2) (h_2-u_1^2)(c_1-c_2)}{u_1^3 h_1 h_2} \varepsilon +O(\varepsilon ^2 ) ,  \nonumber \\
& \psi _n (\overline{f}(z/q ) ,f ;g ) = \frac{(h_1-h_2)(u_1 z-  q h_2 ) (h_2 z - h_1 u_1)}{q h_2 u_1^2 z }+O(\varepsilon  ) , \nonumber \\
& \psi _d (\overline{f}(z/q ) ,f ;g ) = - \frac{(h_1-h_2)(u_1 z-  q h_2 ) (h_2 z - h_1 u_1)}{h_1 ^2 h_2 ^2 z }+O(\varepsilon  ) \nonumber
\end{align}
and
\begin{align}
& \frac{\psi _n (\overline{f}(z/q ) ,f ;g )}{\psi _d (\overline{f}(z/q ) ,f ;g ) } = -\frac{h_1^2 h_2}{q u_1^2} \label{eq:psinpsidratio} \\
& \quad +  \frac{ h_1^2 h_2 }{q u_1^3 } \Bigl\{ \frac{h_2 (h_1-u_1^2) (h_2 + u_1^2 ) (u_1 z-h_1) (z - q u_1) }{u_1^2 (h_1-h_2) (u_1 z - h_2 q) (h_2 z - h_1 u_1)}  (c_1-c_2)+ 2  c_2 \Bigr\}  \varepsilon +O(\varepsilon ^2 ) . \nonumber
\end{align}
Set $u= u_1 +c_2 \varepsilon $.
Then $g= u + h_2 /u$ and it follows from Eq.~(\ref{eq:PnPd}) that
\begin{align}
& V(\overline{f}(z/q ) ,f ; g )  \nonumber \\
& = q \psi _n (\overline{f}(z/q ) ,f ; g ) P_d ( u + h_2 /u ) - h_1^2 h_2^4 \psi _d (\overline{f}(z/q ) ,f ;g ) P_n (h_2 ,u + h_2 / u ) \nonumber \\
& = q \psi _n (\overline{f}(z/q ) ,f ; g ) \frac{u^3 (u^2 -u_1^2  ) U( h_2/ u ) }{u- h_2/u} \nonumber \\
& \quad +q \psi _n (\overline{f}(z/q ) ,f ; g ) \frac{u^3 u_1^2 U( h_2 / u ) }{u- h_2 /u} + h_1^2 h_2 \psi _d (\overline{f}(z/q ) ,f ;g ) \frac{u ^3 U( h_2 /u ) }{u- h_2 /u } \nonumber \\
& \quad - q \psi _n (\overline{f}(z/q ) ,f ; g ) \frac{ ( h_2 /u )^5 U(u)}{u- h_2 /u }  - h_1^2 h_2^4 \psi _d (\overline{f}(z/q ) ,f ;g ) \frac{U(u)/ u^3  }{u- h_2 /u } .  \nonumber 
\end{align}
We define the equivalence by 
\begin{align}
& A(\varepsilon ) \simeq B(\varepsilon ) \; \Leftrightarrow \lim _{\varepsilon \to 0} \frac{A(\varepsilon )}{B(\varepsilon )}=1.  \nonumber 
\end{align}
Then
$U(u) = U(u_1 +\varepsilon c_2 ) \simeq \varepsilon c_2 U_7 (u_1 )$ and $U(h_2 /u ) \simeq - (u_1 - h_2 /u_1 ) U_7 (h_2 /u_1 )$.
Hence 
\begin{align}
& q \psi _n (\overline{f}(z/q ) ,f ; g ) \frac{u^3 (u^2 -u_1^2  ) U( h_2 /u ) }{u- h_2 /u}  \nonumber \\
&  \simeq - 2 q u_1^4 U_7 ( h_2 /u_1 ) c_2  \psi _n (\overline{f}(z/q ) ,f ; g )  \varepsilon \simeq  2 u_1^2 h_1^2 h_2  c_2  U_7 ( h_2 /u_1 ) \psi _d (\overline{f}(z/q ) ,f ; g ) \varepsilon .
\nonumber 
\end{align}
It follows from Eq.~(\ref{eq:psinpsidratio}) that 
\begin{align}
& \psi _n (\overline{f}(z/q ) ,f ; g ) q \frac{u^3 u_1^2 U( h_2 /u ) }{u- h_2 /u} + h_1^2 h_2 \psi _d (\overline{f}(z/q ) ,f ;g ) \frac{u ^3 U( h_2 /u ) }{u- h_2 /u }  \nonumber \\
& \simeq - \psi _d (\overline{f}(z/q ) ,f ; g )  U_7 (h_2 /u_1 ) \nonumber \\
& \quad \cdot \Bigl\{ \frac{h_1^2 h_2^2 (h_1 -u_1^2 ) (h_2 + u_1^2 ) (u_1 z-h_1) (z - q u_1) }{(h_1-h_2) (u_1 z - h_2 q) (h_2 z - h_1 u_1)}  (c_1-c_2) + 2 h_1^2 h_2 u_1 ^2  c_2 \Bigr\}  \varepsilon . \nonumber 
\end{align}
Similarly, we have
\begin{align}
& - q \psi _n (\overline{f}(z/q ) ,f ; g ) \frac{ ( h_2 /u )^5 U(u)}{u- h_2 /u}  - h_1^2 h_2^4 \psi _d (\overline{f}(z/q ) ,f ;g ) \frac{U(u) /u^3 }{u- h_2 / u }  \nonumber \\
& \simeq  q \frac{h_1^2 h_2}{q u_1^2} ( h_2 / u_1 )^5 \varepsilon  c_2 \frac{  U_7 (u_1 )}{u_1 - h_2 /u_1}  \psi _d (\overline{f}(z/q ) ,f ; g )   - h_1^2 h_2^4 \frac{\varepsilon  c_2 }{u_1^3}\frac{ U_7 (u_1 ) }{u_1- h_2 /u_1 } \psi _d (\overline{f}(z/q ) ,f ;g ) \nonumber \\
&  = - \frac{h_1^2 h_2^4 ( u_1 + h_2/ u_1 )}{u_1^5}  U_7 (u_1 )  \psi _d (\overline{f}(z/q ) ,f ; g ) c_2  \varepsilon  . \nonumber 
\end{align}
Thus
\begin{align}
& V(\overline{f}(z/q ) ,f ; g ) \simeq  - \frac{h_1^2 h_2^4 }{u_1^6} ( u_1^2 + h_2 ) U_7 (u_1 )  \psi _d (\overline{f}(z/q ) ,f ; g )  c_2  \varepsilon  \nonumber  \\
& \quad - \psi _d (\overline{f}(z/q ) ,f ; g )  U_7 (h_2 /u_1 ) \frac{h_1^2 h_2^2 (-u_1^2+h_1) (u_1 z-h_1) (z - q u_1) (u_1^2+h_2)}{(h_1-h_2) (u_1 z - h_2 q) (h_2 z - h_1 u_1)}  (c_1-c_2)  \varepsilon \nonumber \\
& \simeq   \frac{ h_2^2 (h_1-h_2) ( u_1^2 + h_2 ) (u_1 z-  q h_2 ) (h_2 z - h_1 u_1)}{ u_1^6 z } U_7 (u_1 ) c_2 \varepsilon  \nonumber   \\
& \quad +  \frac{(-u_1^2+h_1) (u_1^2+h_2) (u_1 z-h_1) (z - q u_1) }{z }  U_7 (h_2 /u_1 ) (c_1-c_2) . \varepsilon . \nonumber 
\end{align}
On the other hand, we have
\begin{align}
& \frac{(h_1-h_2)z^2 (z^2 -q h_1 ) }{q h_1^3 h_2^3 g \varphi (f,g) ( g-g (h_1/z) )( g-g (z/q)) }  \nonumber \\
& \simeq \frac{- z^4 (z^2 -q h_1 ) u_1^6 }{ h_1 h_2^2  (h_2 + u_1^2 )  (h_1 -u_1^2 ) (h_2 -u_1^2) (h_2 z - h_1 u_1) (z - q u_1) (u_1 z-h_1) (u_1 z-h_2 q)  (c_1-c_2)\varepsilon } . \nonumber 
\end{align}
Therefore
\begin{align}
& \frac{(h_1-h_2)z^2 (z^2 -q h_1 )  V( \overline{f}(z/q ) ,f ;g)}{q h_1^3 h_2^3 g \varphi (f,g) ( g-g (h_1/z))( g-g (z/q)) }  \nonumber \\
& \simeq - \frac{(h_1-h_2)  U_7 (u_1 ) z^3 (z^2 -q h_1 ) c_2 }{ h_1 (-u_1^2+h_2) (-u_1^2+h_1) (z - q u_1) (u_1 z-h_1) (c_1-c_2)  }   \nonumber \\
& - \frac{ u_1^6 U_7 (h_2 /u_1 ) z^3 (z^2 -q h_1 )}{ h_1 h_2^2  (-u_1^2+h_2) (h_2 z - h_1 u_1) (u_1 z-h_2 q) } . \nonumber 
\end{align}
By substituting it into Eq.~(\ref{eq:L1E8}), we have
\begin{align}
& \frac{ - q^5 u_1 z U_7(z/q)}{(z^2 - h_1 q^2) (u_1 z-q h_1 ) } y(z/q) + \frac{ q^5 u_1 z U_7(z/q)}{(z^2 - h_1 q^2) } \frac{(h_2 z - q h_1 u_1) }{ h_1 (z - q u_1) (u_1 z-h_2 q)} y(z)  \nonumber \\
& \quad + \frac{z^8 u_1 U_7 (h_1 /z)}{h_1 ^4 (z^2-h_1) (z - u_1) } y(qz) -\frac{z^8 u_1 U_7 (h_1 /z)}{h_1 ^4 (z^2-h_1)} \frac{h_1 (u_1 z-h_2)}{(h_2 z - h_1 u_1) (u_1 z-h_1)} y(z) \nonumber \\
& \quad -  \frac{(h_1-h_2)  U_7 (u_1 ) z^3 (z^2 -q h_1 ) c_2 }{ h_1 (h_1 -u_1^2) (h_2 -u_1^2) (z - q u_1) (u_1 z-h_1) (c_1-c_2)  } y(z)  \nonumber \\
& \quad - \frac{ u_1^6 U_7 (h_2 /u_1 ) z^3 (z^2 -q h_1 )}{ h_1 h_2^2  (h_2 -u_1^2) (h_2 z - h_1 u_1) (u_1 z-h_2 q) }  y(z) = 0 . \nonumber 
\end{align}
Recall that the function $\tilde{y} (z)$ was defined in Eq.~(\ref{eq:yztildeyz}) and it follows that $y(z/q)= z^{-1} q (z/q -q u_1) (u_1 z/q  - h_1) \tilde{y} (z/q)$ and $y(qz)= z^{-1} q^{-1} (qz -q u_1) (q u_1 z  - h_1) \tilde{y} (qz )$.
Then we obtain
\begin{align}
& \frac{ - q^{-3} (z -q^2 u_1) \prod _{j=2}^8 (z - q u_j)}{z^2 (z^2 -q h_1 )(z^2 - h_1 q^2) }  \tilde{y} (z/q) - \frac{ (q u_1 z  - h_1) \prod _{j=2}^8  (u_j z -h_1)}{h_1 ^4 z^2 (z^2 -q h_1 )(z^2-h_1) }  \tilde{y} (qz ) \label{eq:yzqyqzyHz} \\
& + H(z) \tilde{y} (z)  - \frac{(h_1-h_2) c_2 \prod _{j=2}^8  (u_1 -u_j ) }{ u_1 h_1 (h_1 -u_1^2) (h_2 -u_1^2) (c_1-c_2) } \tilde{y} (z) = 0 , \nonumber 
\end{align}
where
\begin{align}
& H(z)= \frac{ q^5 (h_2 z - q h_1 u_1) (u_1 z - h_1) U_7(z/q)}{h_1 z^2 (z^2 -q h_1 )(z^2 - q^2 h_1 ) (u_1 z-h_2 q)} \nonumber \\
& \quad  -\frac{z^5 (u_1 z-h_2) (z-q u_1) U_7 (h_1 /z)}{h_1 ^3 (z^2 -q h_1 ) (z^2-h_1) (h_2 z - h_1 u_1) } - \frac{ u_1^5 U_7 (h_2 /u_1 ) (z-q u_1) (u_1 z - h_1) }{ h_1 h_2^2  (h_2 -u_1^2 ) (h_2 z - h_1 u_1) (u_1 z-q h_2) }  . \nonumber 
\end{align}
It is shown that the poles $z= \pm (q h_1 )^{1/2} , h_1 u_1/h_2, q h_2 /u_1$ of the function $H(z) $ are cancelled, and we have
\begin{align}
& H(z)= \frac{z (h_1^{1/2} - qu_1) (h_1^{1/2} - u_2) (h_1^{1/2} - u_3) \cdots (h_1^{1/2} - u_8) }{2 h_1^{5/2} (z - h_1^{1/2})(z - qh_1^{1/2})} \label{eq:Hz03} \\
& \qquad  -\frac{z (h_1^{1/2} + qu_1) (h_1^{1/2} + u_2) (h_1^{1/2} + u_3) \cdots (h_1^{1/2} + u_8) }{2 h_1^{5/2} (z + h_1^{1/2})(z + qh_1^{1/2})} \nonumber \\
&  \qquad + h_1 h_2 (q+1) (z^2/(q h_1)^2 +z^{-2}) - (z/(q h_1) +z^{-1} ) h_2 ( q u_1 + u_2  \nonumber \\
&  \qquad + u_3 + \cdots + u_8  + h_1/(q u_1) + h_1/u_2 + h_1/u_3 + \cdots + h_1/u_8 ) +C , \nonumber
\end{align}
where $C$ is a constant which does not depend on the variable $z$ and the parameter $c_1$ and $c_2$.
Recall that there is a relation $h_1^2 h_2^2 =q u_1 u_2 u_3 u_4 u_5 u_6 u_7 u_8 $.
By substituting Eq.~(\ref{eq:Hz03}) and $h_2 = ( q u_1 u_2 u_3 u_4 u_5 u_6 u_7 u_8 )^{1/2} /h_1$ into Eq.~(\ref{eq:yzqyqzyHz}), we obtain the theorem.
\end{proof}
Set $z= (qh_1)^{1/2} x $, $\tilde{y} (z) = u(x) $, $u_1 = q^{-1} h_1^{1/2} v_1$,
$u_2 = h_1^{1/2} v_2, \dots ,u_8 = h_1^{1/2} v_8$ and $E=C_0 c_2/(c_1 -c_2 ) + C_0 ' $ in Eq.~(\ref{eq:B-B+B0eq}).
Then we have
\begin{align}
& \frac{ \prod _{j=1}^8 ( x -q^{1/2} v_j )  }{q  x^2 (x^2 -1 )( x^2 - q ) } u(x/q) + \frac{ \prod _{j=1}^8 ( q^{1/2} v_j x - 1)}{ q x^2 (x^2 -1 )(q x^2 - 1) } u(qx) \label{eq:B0eq} \\
& \quad  + \tilde{B}^{0} (x) u(x) = E u(x) , \nonumber 
\end{align}
where 
\begin{align}
& \tilde{B}^{0} (x) = \frac{-q^{1/2} x }{2 (  x - q^{-1/2})( x - q^{1/2})}  \prod _{j=1}^8 (1 - v_j ) + \frac{q^{1/2} x }{2 ( x+ q^{-1/2} )( x + q^{1/2} )} \prod _{j=1}^8 (1 + v_j ) \nonumber \\
& \qquad  - ( v_1 v_2 v_3 v_4 v_5 v_6 v_7 v_8 )^{1/2} \Bigl[ (q+1) \Bigl( x^2 +\frac{1}{ x^2}\Bigr) - q^{1/2} \Bigl( x + \frac{1}{ x} \Bigr) \sum _{j=1}^8 \Bigl(v_j + \frac{1}{v_j } \Bigr) \Bigr] . \nonumber 
\end{align}

We compare it with the firstly degenerated Ruijsenaars-van Diejen operator.
Let $\exp(\pm ia_{-}\partial_{z})$ be the shift operator such that $\exp(\pm ia_{-}\partial_{z}) f(z) = f(z \pm ia_{-}) $.
The firstly degenerated Ruijsenaars-van Diejen operator $\tilde{A}^{\langle 1 \rangle} (h;z)$ was discussed in \cite{vD,TakR}, and it was expressed in \cite{TakR} as 
\begin{equation}
\tilde{A}^{\langle 1 \rangle} (h;z) = \tilde{V}^{\langle 1 \rangle} (h;z)\exp(-ia_{-}\partial_{z})+\tilde{W}^{\langle 1 \rangle}(h;z)\exp(ia_{-}\partial_{z}) + U^{\langle 1 \rangle} (h;z), \nonumber 
\end{equation}
where
\begin{align}
& \tilde{V}^{\langle 1 \rangle} (h;z) = \frac{\prod_{n=1}^8 (1-e^{-2\pi i z}e^{2\pi i \tilde{h}_n }e^{-\pi a_{-}})}{e^{-2\pi a_{-}} e^{-4 \pi i z }(1-e^{-4 \pi i z })(1-e^{-4\pi i z}e^{-2\pi a_{-}})} , \nonumber \\
& \tilde{W}^{\langle 1 \rangle} (h;z) = \frac{\prod_{n=1}^8 (1-e^{2\pi i z}e^{2\pi i \tilde{h}_n }e^{-\pi a_{-}})}{e^{-2\pi a_{-}} e^{4 \pi i z}(1-e^{4 \pi i z})(1-e^{4\pi i z}e^{-2\pi a_{-}})}, \nonumber \\
& U^{\langle 1 \rangle} (h;z) = \frac{\prod_{n=1}^8 (e^{2 \pi i \tilde{h}_n }- 1)}{2(1-e^{2\pi i z}e^{\pi a_- })(1-e^{-2\pi i z}e^{\pi a_- })} + \frac{\prod_{n=1}^8 (e^{2 \pi i \tilde{h}_n } + 1) }{2(1+e^{2\pi i z}e^{\pi a_- })(1+e^{-2\pi i z}e^{\pi a_- })} \nonumber \\
& \qquad \qquad \quad  + e^{-\pi a_- } \prod_{n=1}^8 e^{\pi i \tilde{h}_n } \cdot \Big[ (e^{2\pi i z} +  e^{-2\pi i z} ) \sum _{n=1}^8  ( e^{2\pi i \tilde{h}_n }+ e^{-2\pi i \tilde{h}_n })  \nonumber \\
& \qquad \qquad \qquad \qquad \qquad \qquad \qquad \qquad - (e^{\pi a_- }+e^{-\pi a_- } ) (e^{4\pi i z} + e^{-4\pi i z} ) \Big] . \nonumber 
\end{align}
Set $x= e^{2\pi i z}$, $q= e^{-2\pi a_{-}}$ and $v_n = e^{2\pi i \tilde{h}_n }$.
Then we have $\exp(\pm ia_{-}\partial_{z}) f(e^{2\pi i z}) =  f(e^{2\pi i (z \pm  ia_-)})=  f(q^{\pm 1} x) $. 
It follows from a straightforward calculation that the equation $\tilde{A}^{\langle 1 \rangle} (h;z) f(z)= E f(z) $ is equivalent to Eq.~(\ref{eq:B0eq}).
Therefore we obtain the following theorem.
\begin{theorem}
The $q$-difference equation in Eq.~(\ref{eq:B0eq}) which was obtained from the associated linear equation of the $q$-Painlev\'e equation of type $E^{(1)}_8$ is equivalent to the equation $\tilde{A}^{\langle 1 \rangle} (h;z) f(z)= E f(z) $, where $ \tilde{A}^{\langle 1 \rangle} (h;z) $ is the firstly degenerated Ruijsenaars-van Diejen operator and $E$ is an arbitrary complex number.
\end{theorem}

\section*{Acknowledgements}
The third author was supported by JSPS KAKENHI Grant Number JP18K03378.

\end{document}